\documentclass[12pt]{article}

\usepackage[margin=1in]{geometry}
\usepackage{amsmath,amssymb,amsthm,mathtools}
\usepackage{hyperref}
\usepackage{enumitem}
\usepackage{tikz,float}
\usepackage{tikz-cd,cite}
\usepackage{graphicx}
\usetikzlibrary{arrows.meta,positioning,calc}
\everymath{\displaystyle}
\renewenvironment{proof}{{\noindent\bfseries Proof.}}{\qed}

\theoremstyle{plain}
\newtheorem{theorem}{Theorem}[section]
\newtheorem{proposition}[theorem]{Proposition}
\newtheorem{lemma}[theorem]{Lemma}

\theoremstyle{definition}

\newtheorem{example}[theorem]{Example}
\newtheorem{remark}[theorem]{Remark}
\newtheorem{problem}[theorem]{Problem}
\newtheorem{conjecture}[theorem]{Conjecture}

\newcommand{\KK}{\Bbbk}

\newcommand{\F}{\mathcal{F}}
\newcommand{\IF}{I_{\F}}
\newcommand{\MIN}{\mathrm{MinCov}}

\newcommand{\V}{V}

\newcommand{\A}{A}
\newcommand{\B}{B}

\providecommand{\keywords}[1]
{
	\noindent\textbf{Keywords:} #1
}
\providecommand{\ams}[2]
{
	\noindent\textbf{2020 AMS subject classification:} #2
}

\title{The NF-operator and the NF-Numbers of Simplicial Complexes}
\author{Bilal Ahmad Rather\\
	{\em\small School of Mathematics and Statistics, Shandong University of Technology,}\\
	{\em \small Zibo 255049, China}\\
	\texttt{bilalahmadrr@gmail.com}
}

\date{}

\begin{document}
	\maketitle
	
	\begin{abstract}
		Let $\bigtriangleup$ be a simplicial complex and let $\delta_{\mathcal{NF}}$ denote the NF-operator. The NF-complex $\delta_{\mathcal{NF}}(\bigtriangleup)$ is defined as the Stanley--Reisner complex of the facet ideal of $\bigtriangleup$. Iterating $\delta_{\mathcal{NF}}$ gives a periodic orbit (up to isomorphism), and the smallest positive integer $t$ for which $\delta_{\mathcal{NF}}^{\,t}(\bigtriangleup)\cong \bigtriangleup$ is called the \emph{NF-number} of $\bigtriangleup$ (Habi and Mahmood, Algebra Colloquium, 2022). In this work, we provide various results and determine explicit formulas for the NF-number for several families of graphs. In particular, we compute the NF-number for dumbbell graphs. We also prove that the NF-number of the complete split graph $S_{n,m}$ equals $m+n+2$, and that the NF-number of the double star $D_{p+q}$ equals $p+q+4$. We conclude with remarks, open problems, and conjectures to guide future research.
	\end{abstract}

	
	\ams{}{13F55, 05E45.}

	\section{Introduction}
	
	The \emph{Stanley--Reisner ideal} $I_{\bigtriangleup_{\mathcal{N}}(\bigtriangleup)}$ (encoding nonfaces) and the \emph{facet ideal} $\IF(\bigtriangleup)$ (encoding facets as squarefree monomials) are two classical monomial ideals attached to a simplicial complex $\bigtriangleup$ on a finite set, see \cite{Vizing1965,Reisner1976,Stanley1975}.
	Faridi stressed facet ideals as a suitable complement to Stanley-Reisner theory \cite{Faridi2002}. Hibi and Mahmood \cite{HibiMahmood2020} discovered that passing from $\bigtriangleup$ to a new complex $\delta_{\mathcal{NF}}(\bigtriangleup)$ transforms the facet ideal of $\bigtriangleup$ into the Stanley--Reisner ideal of $\delta_{\mathcal{NF}}(\bigtriangleup)$. Iterating over this operator results in a periodic dynamical system on the finite set of complexes on $[n]=\{1,\dots,n\}$ vertices.	The associated period is the \emph{NF-number}.  
	This point of view places $\delta_{\mathcal{NF}}$ between combinatorial commutative algebra and finite dynamical systems. On the algebraic side, it uses the two standard ways of encoding a simplicial complex by squarefree monomial ideals: the Stanley--Reisner ideal and the facet ideal. On the combinatorial side, the operation can be described in terms of minimal vertex covers, or equivalently minimal transversals, of the facet hypergraph of $\bigtriangleup$. This makes the NF-operator closely related to classical duality phenomena for squarefree monomial ideals, including Alexander duality and the study of edge ideals, cover ideals, and facet ideals; see, for instance, \cite{BrunsHerzog1998,MillerSturmfels2005,Villarreal2001,FranciscoHaVanTuyl2011}. The Stanley--Reisner correspondence has been one of the central tools in combinatorial commutative algebra since the work of Stanley and Reisner. It translates topological and combinatorial properties of $\bigtriangleup$ into algebraic properties of the quotient ring $S/I_{\bigtriangleup}$, such as Cohen--Macaulayness, Betti numbers, and Hilbert series; see \cite{Reisner1976,Stanley1975,Hochster1977,Stanley1996,bilal}. Facet ideals, introduced and systematically studied by Faridi, provide a complementary construction: instead of recording the nonfaces of $\bigtriangleup$, they record the maximal faces themselves \cite{Faridi2002}. The NF-complex connects these two constructions by applying the Stanley--Reisner construction to the facet ideal.
	The NF-number can be defined for any simplicial complex, but graphs provide a concrete example. A finite simple graph $G$ transforms into a one-dimensional simplicial complex with edges as facets. 	For disjoint unions of two cliques $K_n \cup K_m$, Hibi and Mahmood calculated the NF-number and obtained a clear formula \cite{HibiMahmood2020}. The NF-number of two complete graphs joined by a common vertex can be seen in \cite{bilalhafiz}. The NF-number of disjoint union of finite copies of complete graph are given in \cite{bilalhafiz1}.
	The problem of determining NF-numbers is therefore naturally related to the problem of understanding how minimal vertex covers evolve under repeated complementation. Even for graphs, this is a delicate question because the first NF-iterate of a graph need not remain one-dimensional. Thus, graph-theoretic data may produce higher-dimensional simplicial complexes after only one or two applications of $\delta_{\mathcal{NF}}$. This phenomenon is one of the reasons explicit NF-number computations are nontrivial beyond complete graphs and their close relatives. In the present work, we continue this line of investigation for several familiar graph families. We first analyze dumbbell graphs, obtained by joining two complete graphs by a bridge edge. We then treat complete split graphs and double stars. These families are sufficiently structured to allow explicit descriptions of their minimal vertex covers, but they are rich enough to show how quickly NF-iteration leaves the category of graphs and enters the broader category of simplicial complexes. Throughout the paper, graphs are regarded as one-dimensional simplicial complexes whose facets are their edges.
	
	 	A \emph{simplicial complex} $\bigtriangleup$ on finite set $\V$ is a family of subsets of $\V$
		such that if $F\in\bigtriangleup$ and $G\subseteq F$, then $G\in\bigtriangleup$.
		The elements of $\bigtriangleup$ are called \emph{faces}, and the maximal faces under inclusion are known as \emph{facets}.
		The collection of facets of $\bigtriangleup$ is denoted by $\F(\bigtriangleup)$. Two complexes $\bigtriangleup,$ and $\Gamma$ on vertex sets $\V, $ and $V'$ are \emph{isomorphic} if there exists a bijection	$\pi:\V\to\V'$ such that $F\in\bigtriangleup$ if and only if $\pi(F)\in\Gamma$.  Consider a field $\KK$ and let $S=\KK[x_v: v\in \V]$ be a polynomial ring.   The \emph{facet ideal} $\IF(\bigtriangleup)\subset S$ is defined as
		\[
		\IF(\bigtriangleup)  =  \bigl(   \prod_{v\in F} x_v  : F\in\F(\bigtriangleup) \bigr).
		\]
	 For a  squarefree monomial ideal $I\subset S$, the \emph{Stanley--Reisner complex} $\bigtriangleup_{\mathcal{N}}(I)$ is defined as
		\[
		\bigtriangleup_{\mathcal{N}}(I) = \{\,F\subseteq \V  : \prod_{v\in F} x_v \notin I\,\}.
		\] 
		 The \emph{NF-complex} (see, \cite{HibiMahmood2020}) of $\bigtriangleup$ is defined as
		\[
		\delta_{\mathcal{NF}}(\bigtriangleup) := \bigtriangleup_{\mathcal{N}}(\IF(\bigtriangleup)),
		\]
		where $\delta_{\mathcal{NF}}$ is the NF-operator.
		Let $\delta_{\mathcal{NF}}^{(0)}(\bigtriangleup)=\bigtriangleup$ and $\delta_{\mathcal{NF}}^{(k)}(\bigtriangleup)=\delta_{\mathcal{NF}}\!\left(\delta_{\mathcal{NF}}^{(k-1)}(\bigtriangleup)\right)$, where $k$ is a positive integer. The \emph{NF-number} of $\bigtriangleup$ is the smallest integer $q>0$, such that $\delta_{\mathcal{NF}}^{(q)}(\bigtriangleup)\cong \bigtriangleup$. \medskip

		A simple graph is an ordered pair $G=(V,E)$, where $V$ is a nonempty set of \emph{vertices} (or nodes) and $E \subseteq \{\{u,v\}\mid u,v \in V, u\neq v\}$ is a set of \emph{edges} representing connections between vertices. The \emph{degree} of a vertex $v\in V$, denoted by $d(v)$, is the number of edges incident to $v$ and is a key local parameter used in analyzing network connectivity. An edge between two distinct vertices $u$ and $v$ is denoted by $uv$ or $\{u,v\}.$ A \emph{vertex cover} is a subset $C \subseteq V$ such that every edge in $E$ has at least one endpoint in $C$, and the minimum size of such a set is called the \emph{vertex cover number}, denoted by $\vartheta(G)$. A \emph{dominating set} is a subset $D \subseteq V$ such that every vertex in $V\setminus D$ is adjacent to at least one vertex in $D$, and the minimum cardinality of a dominating set is the \emph{domination number}, denoted $\varUpsilon(G)$, see \cite{West2001}. 
		
	\medskip

	The paper is organized as: In Section \ref{section 2}, we discuss basics of $\delta_{\mathcal{NF}}$, vertex covers and some existing results. Section \ref{section 3} presents simplicial complexes, results related to NF-numbers of dumbbell graphs, and their other properties.  Section \ref{section 4} discuss the  NF-number of complete split graph $S_{n,m}$, and we show that  its NF-number is $n+m+2$. Section \ref{section 5} gives the NF-number of double star $D_{p,q}$, and in general we prove that $\mathcal{NF}(D_{p,q})=p+q+4.$ We end the article with conclusion and future research directions.
	\section{Vertex covers and a working formula for $\delta_{\mathcal{NF}}$}\label{section 2}
	At first, a question comes in mind, weather such period $q$ exist, and  $\delta_{\mathcal{NF}}^{(q)}(\bigtriangleup)\cong \bigtriangleup$. The following lemma shows the existence of period $q$  \cite{HibiMahmood2020}, and hence, the NF-number of $\bigtriangleup$ is well-defined.
	\begin{lemma}
		For each simplicial complex $\bigtriangleup$ on a finite vertex set, there exists $q>0$ such that $\delta_{\mathcal{NF}}^{(q)}(\bigtriangleup)=\bigtriangleup$.
	\end{lemma}
	
	\begin{proof}
		The sequence $\bigtriangleup,\delta_{\mathcal{NF}}(\bigtriangleup),\delta_{\mathcal{NF}}^{(2)}(\bigtriangleup),\dots$ must ultimately recur as there are only finitely many simplicial complexes on a fixed finite vertex set.
		Hibi and Mahmood \cite[Lemma 1.4]{HibiMahmood2020} prove a stronger injectivity property that: if $\delta_{\mathcal{NF}}(\bigtriangleup')=\delta_{\mathcal{NF}}(\bigtriangleup'')$, then $\bigtriangleup'=\bigtriangleup''$. As a result, the first repetition must return to the beginning point, yielding a valid period.
	\end{proof}
	
	\medskip
	A key computational tool is the characterization of facets of $\delta_{\mathcal{NF}}(\bigtriangleup)$ in terms of minimal vertex covers.
		For a simplicial complex $\bigtriangleup$ on $\V$ with facets $\F(\bigtriangleup)$, a \emph{vertex cover} of $\bigtriangleup$ is a subset $C\subseteq \V$, if $C\cap F\neq\varnothing$, for every $F\in\F(\bigtriangleup)$. If it has no proper subset that is also a vertex cover, then it is \emph{minimal vertex cover}.
		 The set of minimal vertex covers is denoted by $\MIN(\bigtriangleup)$.

\medskip
	The following result shows that the facets of $\delta_{\mathcal{NF}}(\bigtriangleup)$ are complements of minimal covers.
	\begin{lemma}\label{lem:facets-complements}
		Let $\bigtriangleup$ be a simplicial complex on vertex set $\V$. Then
		\[
		\F(\delta_{\mathcal{NF}}(\bigtriangleup)) = \{\,\V\setminus C  : C\in \MIN(\bigtriangleup)\,\}.
		\]
	\end{lemma}
	
	\begin{proof}
		In facet-ideal theory, it is well-known that $\IF(\bigtriangleup)$ has a (irreducible) decomposition \cite[p. 12]{HerzogHibi2011}
		\[
		\IF(\bigtriangleup) = \bigcap_{C\in\MIN(\bigtriangleup)} (x_v : v\in C).
		\]
		A set $F\subseteq\V$ is in $\bigtriangleup_{\mathcal{N}}(\IF(\bigtriangleup))$ if and only if $\prod_{v\in F}x_v \notin \IF(\bigtriangleup)$, meaning it does not belong to at least one prime $(x_v:v\in C)$ in the intersection. This happens precisely when $F\cap C=\varnothing$ for some $C\in\MIN(\bigtriangleup)$. 
		Maximal such $F$ are the exact complements of minimal covers.
	\end{proof}
	
	\medskip
	Also, if $\bigtriangleup\cong\varGamma$ then $\delta_{\mathcal{NF}}^{(k)}(\bigtriangleup)\cong\delta_{\mathcal{NF}}^{(k)}(\varGamma)$ for all $k$, hence $\mathcal{NF}(\bigtriangleup)=\mathcal{NF}(\varGamma)$. An isomorphism renames variables in $S$ and maps facet ideals and Stanley-Reisner complexes to each other.
	
	\medskip
	The following result gives the dimension of $\delta_{\mathcal{NF}}(\bigtriangleup)$  via minimum cover size.
	\begin{proposition}\label{prop:dim}
		Let $\vartheta(\bigtriangleup)=\min\{|C|:C\in\MIN(\bigtriangleup)\}$ be the vertex cover number of $\bigtriangleup$.
		Then $\dim \delta_{\mathcal{NF}}(\bigtriangleup) = |\V|-\vartheta(\bigtriangleup)-1.$ 
	\end{proposition}
	
	\begin{proof}
		By Lemma~\ref{lem:facets-complements}, facets of $\delta_{\mathcal{NF}}(\bigtriangleup)$ have sizes $|\V|-|C|$, for $C\in\MIN(\bigtriangleup)$.
		The largest facet size is $|\V|-\vartheta(\bigtriangleup)$, so the dimension is one less.
	\end{proof}

	\section{Graphs as simplicial complexes}\label{section 3}
	
	 Let $G=(\V,E)$ be a finite simple graph. Regard $G$ as a simplicial complex $\bigtriangleup(G)$ of dimension $1$ with facets as edges	$\F(\bigtriangleup(G))=\{\{u,v\}:uv\in E\}$.
	Then a vertex cover of $\bigtriangleup(G)$ is exactly a vertex cover of the graph $G$. Lemma~\ref{lem:facets-complements} states that facets of $\delta_{\mathcal{NF}}(G)$ are complements of minimal graph vertex covers of $G$. Thus, $\delta_{\mathcal{NF}}(G)$ is usually not $1$-dimensional, and it can have facets of various sizes. 
	
	\medskip
	 According to Hibi and Mahmood \cite{HibiMahmood2020}, the NF-number of the complete graph $K_n$ is $n+1$, while the disjoint union $K_n\cup K_m$ with $n,m\ge 2$ and $(n,m)\neq(2,2)$ has an NF-number of $n+m+2$. The NF-number of the two complete graphs with their one fused vertex and that of disjoint union of cliques can be see in \cite{bilalhafiz,bilalhafiz1}. It is very non-trivial to find NF-number of a general graph. With motivation from the above authors, we carry this work forward for several well known families of graphs.

	\medskip
	  For integers \(n,m\ge 2\), let \(\A=\{x_1,\dots,x_n\}\) and \(\B=\{y_1,\dots,y_m\}\) be disjoint sets, denote by \(K(\A)\) the clique on \(\A\) and by \(K(\B)\) the clique on \(\B\), and choose any \(x_\alpha\in \A\) and any \(y_\beta\in \B\). The graph that results from adding  edges $x_\alpha y_\beta$ to the disjoint union $K(\A)\cup K(\B)$ is called the \emph{bridge graph} $B_{n,m}^{\alpha,\beta}$. If we add a single edge, that is,  when $\alpha=\beta=1$, then $B_{n,m}$ is called the \emph{dumbbell graph}. A dumbbell graph $B_{n,m}$ is depicted in Figure \ref{dumbell}, with solid edges denoting fully drawn adjacency and dotted edges denoting additional clique edges. Although, some authors denote dumbbell graph by $D_{n,m}$, but we keep $ B_{n,m}$, since it is special case of bridge graph $B_{n,m}^{\alpha,\beta}$.
	
	Since any two options $(\alpha,\beta)$ and $(\alpha',\beta')$ result in isomorphic graphs, the position of the bridge is irrelevant. Each clique's permutations convey $x_\alpha\mapsto x_{\alpha'}$ and $y_\beta\mapsto y_{\beta'}$. Therefore, $\mathcal{NF}(B_{n,m}^{\alpha,\beta})$ depends only on $(n,m)$.
	
	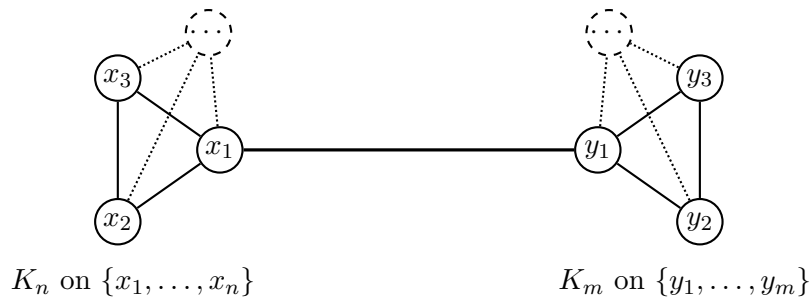
\begin{figure}[h]
		\centering
		\begin{tikzpicture}[
			scale=1,
			vertex/.style={circle, draw, thick, minimum size=6mm, inner sep=0pt, font=\small},
			edge/.style={draw, thick},
			bridge/.style={draw, very thick},
			ghost/.style={circle, draw, thick, dashed, minimum size=6mm, inner sep=0pt, font=\small},
			lab/.style={font=\small}
			]
			
			\node[vertex] (x1) at (0,0) {$x_1$};
			\node[vertex] (x2) at (-1.35,0.95) {$x_3$};
			\node[vertex] (x3) at (-1.35,-0.95) {$x_2$};
			\node[ghost]  (xk) at (-0.15,1.55) {$\cdots$};
			
			\draw[edge] (x1)--(x2);
			\draw[edge] (x1)--(x3);
			\draw[edge] (x2)--(x3);
			\draw[edge, densely dotted] (x2)--(xk);
			\draw[edge, densely dotted] (x3)--(xk);
			\draw[edge, densely dotted] (x1)--(xk);
			
			\node[vertex] (y1) at (5,0) {$y_1$};
			\node[vertex] (y2) at (6.35,0.95) {$y_3$};
			\node[vertex] (y3) at (6.35,-0.95) {$y_2$};
			\node[ghost]  (yk) at (5.15,1.55) {$\cdots$};
			
			\draw[edge] (y1)--(y2);
			\draw[edge] (y1)--(y3);
			\draw[edge] (y2)--(y3);
			\draw[edge, densely dotted] (y2)--(yk);
			\draw[edge, densely dotted] (y3)--(yk);
			\draw[edge, densely dotted] (y1)--(yk);
			
			\draw[bridge] (x1)--(y1);
			
			\node[lab] at (-1.15,-1.75) {$K_n$ on $\{x_1,\dots,x_n\}$};
			\node[lab] at (6.15,-1.75) {$K_m$ on $\{y_1,\dots,y_m\}$};

		\end{tikzpicture}
		\caption{Dumbbell graph $B_{n,m}$: two cliques joined by the bridge $x_1y_1$.}
		\label{dumbell}
	\end{figure}

	 Throughout, we write the vertex set of $B_{n,m}$ as $\V=\A\cup\B$. The following is a well known fact and can be seen in any standard textbook.	
	\begin{lemma}
		In a clique $K_{r}$ with $r\ge 2$ vertices, the minimal vertex covers are exactly the complements of single vertices and so have size $r-1$.
	\end{lemma}

	The next lemma gives the minimal vertex covers of dumbbell graph $B_{n,m}$.
	\begin{lemma}\label{lem:mincov-dumbbell}
		Let $G=B_{n,m}$ with its vertex set written as
		$\A=\{x_1,\dots,x_n\}$, $\B=\{y_1,\dots,y_m\}$, and the unique bridge edge is $x_1y_1$.
		Then a subset $C\subseteq \V$ is a minimal vertex cover of $G$ if and only if
		$C  =  (\A\setminus\{x_i\})  \cup (\B\setminus\{y_j\})$ for some pair $(i,j)\neq(1,1).$ 
	\end{lemma}
	
	\begin{proof} Assume $C$ is a minimal vertex cover of $G$. We consider the following cases.
		
		\smallskip\noindent
		\textbf{Claim 1.} $|C\cap \A|\ge n-1$ and $|C\cap \B|\ge m-1$.
		Now, to prove it, suppose that $|C\cap \A|\le n-2$. Then there exist two distinct vertices $x_r,x_s\in \A\setminus C$.
		Since $\A$ is a clique, the edge $x_rx_s\in E(G)$, but neither endpoint lies in $C$. So, the edge $x_rx_s$ is not covered, which contradicts the claim that $C$ is a vertex cover. Thus, we must have $|C\cap \A|\ge n-1$. A similar reason is valid for  $\B$. 
		
		\smallskip\noindent
		\textbf{Claim 2.} There exist indices $i\in\{1,\dots,n\}$ and $j\in\{1,\dots,m\}$, such that
		 $C\cap \A=\A\setminus\{x_i\}$ and $C\cap \B=\B\setminus\{y_j\}.$  For proof, we note by Claim 1, $C\cap \A$ contains at least $n-1$ vertices of $\A$. Hence,  either $C\cap \A=\A$ or $C\cap \A=\A\setminus\{x_i\}$ for a unique $i$.
		 First, we show that $C\cap \A\neq \A$. Suppose, for contradiction, that $C\cap \A=\A$. Since $n\ge 2$, choose a vertex $x_t\in \A$ with $t\neq 1$, and set $C'=C\setminus\{x_t\}$. We claim that $C'$ is still a vertex cover of $G$.
		 Indeed, every edge inside the clique on $\A$ that is incident to $x_t$ has its other endpoint in $\A\setminus\{x_t\}\subseteq C'$. All other edges inside $\A$ are unaffected. The bridge edge $x_1y_1$ is also unaffected, since $x_t\neq x_1$. Edges inside the clique on $\B$ are unchanged as well. Hence, every edge of $G$ remains covered by $C'$, contradicting the minimality of $C$. Therefore $C\cap \A\neq \A$, and so
		$
		C\cap \A=\A\setminus\{x_i\}
		$ 	for some $i\in\{1,\dots,n\}$.
		
		Similarly, by Claim 1, either $C\cap \B=\B$ or $C\cap \B=\B\setminus\{y_j\}$ for a unique $j$. If $C\cap \B=\B$, then, since $m\ge 2$, choose $y_s\in \B$ with $s\neq 1$ and set $C''=C\setminus\{y_s\}$. Every edge inside the clique on $\B$ incident to $y_s$ remains covered by its other endpoint, and the bridge edge $x_1y_1$ is unaffected because $y_s\neq y_1$. All other edges are unchanged. Thus, $C''$ is still a vertex cover, contradicting the minimality of $C$. Therefore, we get 
		$
		C\cap \B=\B\setminus\{y_j\}
		$
		for some $j\in\{1,\dots,m\}$. 
		
		\smallskip\noindent
		\textbf{Claim 3.} $C_{i,j}$ is a vertex cover. Every edge of $G$ is either:
		(i) an edge inside the clique on $\A$,
		(ii) an edge inside the clique on $\B$, or
		(iii) the bridge edge $x_1y_1$.
		For (i), any edge inside the clique on $\A$ has endpoints $x_r,x_s$; since $C_{i,j}$ omits only $x_i$ from $\A$,
		at least one of $x_r,x_s$ is in $\A\setminus\{x_i\}\subseteq C_{i,j}$, so the edge is covered.
		Similarly, all edges of type (ii) are covered.
		For (iii), as $(i,j)\neq(1,1)$, at least one of $i\neq 1$ or $j\neq 1$ holds, hence $x_1\in C_{i,j}$ or $y_1\in C_{i,j}$.
		Therefore, the bridge edge is covered. Thus, $C_{i,j}$ is a vertex cover. 
		
		\smallskip\noindent
		\textbf{Claim 4.} $C_{i,j}$ is minimal. Let $v\in C_{i,j}$. Then we show that $C_{i,j}\setminus\{v\}$ is not a vertex cover.
		 If $v\in \A\setminus\{x_i\}$, then $v$ is adjacent to $x_i$, since $\A$ is a clique.
		Moreover, by construction $x_i\notin C_{i,j}$, and hence, the edge $vx_i$ has neither endpoint in $C_{i,j}\setminus\{v\}$,
		so it becomes uncovered. Thus ,removing $v$ destroys the vertex cover property.
		 	If $v\in \B\setminus\{y_j\}$, similarly the edge $vy_j$ (which lies in the clique on $\B$) becomes uncovered upon removing $v$,
		since $y_j\notin C_{i,j}$. In either case, no vertex can be removed while preserving the vertex cover property. Hence, $C_{i,j}$ is minimal. Thus,  $C_{i,j}\in \MIN(G)$, completing the proof.
	\end{proof}
	
	\medskip
	The following result gives the first iterate $\delta_{\mathcal{NF}}(B_{n,m})$ of the dumbbell graph.	
	\begin{theorem}\label{thm:delta1}
		If $G=B_{n,m}$ is a dumbbell graph, then $\delta_{\mathcal{NF}}(G)$ is a graph whose edges are exactly
		$E(\delta_{\mathcal{NF}}(G))  =  \bigl\{  x_i y_j  :  1\le i\le n, 1\le j\le m, (i,j)\neq(1,1) \bigr\}.$
		Equivalently, $\delta_{\mathcal{NF}}(G)\cong K_{n,m}\setminus\{x_1y_1\}$.
	\end{theorem}
	
	\begin{proof}
		 By Lemma~\ref{lem:facets-complements}, the facets of $\delta_{\mathcal{NF}}(\bigtriangleup(G))$ are exactly the complements of the minimal
		vertex covers of $\bigtriangleup(G)$ (equivalently, of $G$), that is,
		\[
		\F\bigl(\delta_{\mathcal{NF}}(G)\bigr)=\{ \V\setminus C  : C\in \MIN(G) \}.
		\]
		By Lemma~\ref{lem:mincov-dumbbell}, the minimal vertex covers of $G$ are precisely the sets
		 $C_{i,j}  =  (\A\setminus\{x_i\})\cup(\B\setminus\{y_j\})$, for $(i,j)\neq(1,1).$ For such a pair $(i,j)\neq(1,1)$, we compute the complement in $\V=\A\cup\B$
		\[
		\V\setminus C_{i,j}
		=
		(\A\cup\B)\setminus\bigl((\A\setminus\{x_i\})\cup(\B\setminus\{y_j\})\bigr)
		=
		\{x_i\}\cup\{y_j\}
		=
		\{x_i,y_j\}.
		\]
		Thus, every facet of $\delta_{\mathcal{NF}}(G)$ is a $2$-element set, hence $\delta_{\mathcal{NF}}(G)$ is $1$-dimensional (that is a graph), and its edge set is $E(\delta_{\mathcal{NF}}(G))=\bigl\{\{x_i,y_j\} : (i,j)\neq(1,1)\bigr\}.$ Finally, $E(\delta_{\mathcal{NF}}(G))$ has all cross edges between $\A$ and $\B$ except $\{x_1,y_1\}$, and contains no edges within $\A$ or $\B$. Thus, $\delta_{\mathcal{NF}}(G)$ is the complete bipartite graph $K_{n,m}$ with a single missing edge $\{x_1,y_1\}$, (see Figure \ref{dumbell1}), or equivalently, $\delta_{\mathcal{NF}}(G) \cong K_{n,m}\setminus\{x_1y_1\}.$ 
	\end{proof}

	\begin{figure}[H]
		\centering
		\begin{tikzpicture}[scale=1, every node/.style={circle, draw, inner sep=1.5pt}]
			\node (x1) at (0,1.2) {$x_1$};
			\node (x2) at (0,0.4) {$x_2$};
			\node (x3) at (0,-0.4) {$\vdots$};
			\node (xn) at (0,-1.2) {$x_n$};
			
			\node (y1) at (4,1.2) {$y_1$};
			\node (y2) at (4,0.4) {$y_2$};
			\node (y3) at (4,-0.4) {$\vdots$};
			\node (ym) at (4,-1.2) {$y_m$};
			
			\foreach \xx in {x1,x2,xn}{
				\foreach \yy in {y2,ym}{
					\draw (\xx)--(\yy);
				}
			}
			\draw (x2)--(y1);
			\draw (xn)--(y1);
			\draw[densely dashed] (x1)--(y1); 
			
			\node[draw=none, rectangle, inner sep=0pt] at (2,-2.0) {\small $\delta_{\mathcal{NF}}(B_{n,m}) \cong K_{n,m}$ with the single edge $x_1y_1$ missing.};
		\end{tikzpicture}
		\caption{Schematic of $\delta_{\mathcal{NF}}(B_{n,m})$. The dashed edge is \emph{not} present.}
		\label{dumbell1}
	\end{figure}
	
	\medskip
	 The second iteration is already higher-dimensional, with an especially stiff facet structure. First, we find the minimal vertex cover of the graph shown in Figure \ref{dumbell1}.
	\begin{lemma}\label{lem:mincov-bip-minus}
		Let $\A=\{x_1,\dots,x_n\}$ and $\B=\{y_1,\dots,y_m\}$ with $n,m\ge 2$.
		Let $H$ be the bipartite graph on $\A\cup\B$ whose edge set is
		\[
		E(H)=\bigl\{ x_i y_j: 1\le i\le n, 1\le j\le m, (i,j)\neq(1,1) \bigr\}
		=E(K_{n,m})\setminus\{x_1y_1\}.
		\]
		The minimal vertex covers of $H$ are precisely the following three sets $\A,  \B,$ and $ (\A\setminus\{x_1\})\cup(\B\setminus\{y_1\}).$ 
	\end{lemma}
	
	\begin{proof}
		We verify that each of the three sets stated in assertion is a minimal vertex cover, and then show that no more minimal vertex covers exist. We have the following steps to consider.
		
		\smallskip\noindent
		\textbf{1.} $\A$ and $\B$ are minimal vertex covers. Every edge of $H$ has one endpoint in $\A$ and one endpoint in $\B$, therefore $\A$ meets every edge, resulting in a vertex cover. Similarly, $\B$ represents a vertex cover.
		 To demonstrate that $\A$ is minimum, consider $\A'=\A\setminus\{x_i\}$.
		Since $m\ge 2$, select an index $j\in\{1,\dots,m\}$ with $j\neq 1$ (for instance, $j=2$).
		Then $(i,j)\neq (1,1)$, so $x_i y_j\in E(H)$. However, since $x_i\notin \A'$ and $\A'$ has no $y$-vertices at all, $\A'$ does not intersect the edge $x_i y_j$. Consequently, $\A$ is minimum since $\A'$ is not a vertex cover.
		The same reasoning (with $n\ge 2$) demonstrates that $\B$ is minimal, since eliminating any $y_j$ leaves uncovered an edge $x_2y_j$.
		
		\smallskip\noindent
		\textbf{2.} $C_0:=(\A\setminus\{x_1\})\cup(\B\setminus\{y_1\})$ is a minimal vertex cover.  For	cover property, let $e=x_i y_j$ be any arbitrary edge in $H$. Then $(i,j)\neq(1,1)$, which implies that either $i\neq 1$ or $j\neq 1$ is true.
		If $i\neq 1$, then $x_i\in \A\setminus\{x_1\}\subseteq C_0$, which means $e$ is covered.
		If $i=1$, then necessarily $j\neq 1$ (since $(i,j)\neq(1,1)$), and thus $y_j\in \B\setminus\{y_1\}\subseteq C_0$,
		so again $e$ is covered. Hence $C_0$ is a vertex cover.\\
		For minimality, we demonstrate how deleting any vertex from $C_0$ breaks the cover property.
		If $v=x_i$ with $i\ge 2$, then consider the edge $x_i y_1$. Since $i\ge 2$, we have $(i,1)\neq (1,1)$, so $x_i y_1\in E(H)$.
		However, $y_1\notin C_0$, and if we remove $x_i$, neither endpoint of $x_i y_1$ lies in $C_0\setminus\{x_i\}$, leaving the edge uncovered.
		Similarly, if $v=y_j$ with $j\ge 2$, then the edge $x_1 y_j$ belongs to $E(H)$ (as $(1,j)\neq(1,1)$), and since $x_1\notin C_0$,
		removing $y_j$ leaves $x_1 y_j$ uncovered. So, $C_0$ is minimal.
		
		\smallskip\noindent
		\textbf{3.} No other minimal vertex covers exist. Let $C$ be a minimal vertex cover of $H.$ 	If $C=\A$ or $C=\B$, we are done. So, assume that $C\neq \A$ and $C\neq \B$.
		Then $C$ omits at least one vertex from $\A$ and at least one vertex from $\B$, that is, there exist $x_i\in \A\setminus C$ and $y_j\in \B\setminus C$.		We claim that necessarily $i=1$ and $j=1$.
		 If $i\neq 1$, then $x_i$ is adjacent in $H$ to every vertex of $\B$, since the only missing edge is $x_1y_1$.
		Since $x_i\notin C$ and $C$ must meet every edge incident to $x_i$, it follows that \emph{every} $y\in\B$ must lie in $C$.
		Hence $\B\subseteq C$, which forces $C=\B$ by minimality, since $\B$ is itself a vertex cover,		contradicting our assumption $C\neq \B$. Therefore, $i=1$, so $x_1\notin C$. By symmetry, if $j\neq 1$, $y_j$ is adjacent to every vertex of $\A$, hence $\A\subseteq C$, requiring $C=\A$ and contradicting $C\neq\A$.	Hence $j=1$, so $y_1\notin C$.
		 Now, since $x_1\notin C$, every edge $x_1y_\ell$ that exists in $H$ must be covered by $y_\ell\in C$.
		However, the edges incident to $x_1$ in $H$ are precisely $x_1y_\ell$ for $\ell\ge 2$ (the edge $x_1y_1$ is missing), therefore $y_\ell\in C$ for all $\ell\ge 2$.
		Similarly, since $y_1\notin C$, every edge $x_k y_1$ for $k\ge 2$ must be covered by $x_k\in C$, so $x_k\in C$ for all $k\ge 2$.
		Therefore, we have $(\A\setminus\{x_1\})\cup(\B\setminus\{y_1\}) \subseteq C.$ 	However, the left-hand side is already a vertex cover (step 2), therefore by minimality of $C$, we must have equality
		 $C=(\A\setminus\{x_1\})\cup(\B\setminus\{y_1\})=C_0.$ 
		This demonstrates that the only minimal vertex covers are $\A$, $\B$, and $C_0$ as claimed.
	\end{proof}
	
	\medskip
	Now, we have the following result, which presents the second NF-iterate of the dumbbell.
	\begin{theorem}\label{thm:delta2}
		Let $G=B_{n,m}$ on $\V=\A\cup\B$. Then $\delta_{\mathcal{NF}}^{(2)}(G)$ has exactly three facets
		 $\F\!\left(\delta_{\mathcal{NF}}^{(2)}(G)\right) = \{\A, \B, \{x_1,y_1\}\}.$ 
	\end{theorem}
	
	\begin{proof}
		As $G$  is a $1$-dimensional simplicial complex $\bigtriangleup(G)$ with facets as edges. So, by Theorem~\ref{thm:delta1}, the first NF-iterate is the graph $\delta_{\mathcal{NF}}(G)=H,$ where $H \cong K_{n,m}\setminus\{x_1y_1\},$ that is, $H$ is the bipartite graph on the same vertex set $\V=\A\cup\B$ with edge set
		 $E(H)=\{x_i y_j: 1\le i\le n, 1\le j\le m, (i,j)\neq(1,1)\}.$ 
		Lemma~\ref{lem:mincov-bip-minus} states that the set of minimal vertex covers of $H$ is
		\[
		\MIN(H)=\Bigl\{ \A, \B, (\A\setminus\{x_1\})\cup(\B\setminus\{y_1\}) \Bigr\}.
		\]
		 By Lemma~\ref{lem:facets-complements}, for any simplicial complex (hence in particular for the graph complex $\bigtriangleup(H)$),
		the facets of its NF-complex are exactly complements of minimal vertex covers
		\[
		\F(\delta_{\mathcal{NF}}(H))=\{ \V\setminus C  : C\in\MIN(H) \}.
		\]
		Since $\delta_{\mathcal{NF}}(H)=\delta_{\mathcal{NF}}^{(2)}(G)$, it remains to compute these three complements inside $\V=\A\cup\B$. (i) If $C=\A$, then $\V\setminus C=(\A\cup\B)\setminus\A=\B$. (ii) If $C=\B$, then $\V\setminus C=(\A\cup\B)\setminus\B=\A$. (iii)   If $C=(\A\setminus\{x_1\})\cup(\B\setminus\{y_1\})$, then using that $\A$ and $\B$ are disjoint,
			\[
			\V\setminus C
			= (\A\cup\B)\setminus\bigl((\A\setminus\{x_1\})\cup(\B\setminus\{y_1\})\bigr)
			= \{x_1\}\cup\{y_1\}
			= \{x_1,y_1\}.
			\] 
		Therefore, we obtain
		\[
		\F\!\left(\delta_{\mathcal{NF}}^{(2)}(G)\right)
		=\F(\delta_{\mathcal{NF}}(H))
		=\{\B, \A, \{x_1,y_1\}\}
		=\{\A, \B, \{x_1,y_1\}\}.
		\]
	\end{proof}
	
	The following is the immediate structural consequences of the above result.
	\begin{remark}
		If $\max\{n,m\}\ge 3$, then $\delta_{\mathcal{NF}}^{(2)}(B_{n,m})$ has a facet of size at least $3$, hence $\dim \delta_{\mathcal{NF}}^{(2)}(B_{n,m})\ge 2$.
		Thus, the orbit leaves the world of graphs at step $2$.
	\end{remark}

	Theorems ~\ref{thm:delta1} and  \ref{thm:delta2} provide the first two iterates in closed form.
	A natural follow-up question is the period of the complete orbit.  We have the following sharp conjectural formula for the NF-number of the dumbbell graph.
	\begin{conjecture}\label{conj:nf-dumbbell}
		Let $n,m\ge 2$ and let $G=B_{n,m}$.
		Then
		\[
		\mathcal{NF}(G)=
		\begin{cases}
			2, & (n,m)=(2,2),\\
			n+m+2, & \text{otherwise}.
		\end{cases}
		\]
	\end{conjecture}
	
	 The exceptional case $(n,m)=(2,2)$ is the path $P_4$.  According to Hibi and Mahmood \cite{HibiMahmood2020}, $P_4$ has an NF-number of $1$ up to isomorphism. However, when the labeled vertex set is fixed, the orbit can exhibit a 2-cycle. Both opinions exist in the literature, depending on whether intermediate complexes are identified and relabeled at each step or at the conclusion.		In this study, we use the return up to isomorphism convention for the NF-number. According to this convention, Conjecture~\ref{conj:nf-dumbbell} predicts the smallest $q$ with $\delta_{\mathcal{NF}}^{(q)}(G)\cong G$.
	
	\medskip
	Computational evidence indicates that Conjecture~\ref{conj:nf-dumbbell} is valid for all $2\le n,m\le 5$. Using Lemma~\ref{lem:facets-complements}, a direct implementation of $\delta_{\mathcal{NF}}$ counts minimal vertex coverings and iterates the operator until the first return.		Exhaustive checks (Table \ref{tab 1}) for all $(n,m)$ with $2\le n,m\le 5$ match the predicted values
		$\mathcal{NF}(B_{n,m})=n+m+2$ except for $(2,2)$.
	 \begin{table}[h]
		\centering
		\begin{tabular}{c||cccc}
			$\mathcal{NF}(B_{n,m})$ & $m=2$ & $m=3$ & $m=4$ & $m=5$\\ \hline\hline
			$n=2$ & $2$ & $7$ & $8$ & $9$\\
			$n=3$ & $7$ & $8$ & $9$ & $10$\\
			$n=4$ & $8$ & $9$ & $10$ & $11$\\
			$n=5$ & $9$ & $10$ & $11$ & $12$\\ \hline\hline
		\end{tabular}
		\caption{Computed NF-numbers for $2\le n,m\le 5$, consistent with $n+m+2$ except at $(2,2)$.}
		\label{tab 1}
	\end{table}
	
	Next, we have the partial results toward Conjecture \ref{conj:nf-dumbbell}.  Even without a complete closed-form description of each iterate, nontrivial constraints can be proven.
	The following result gives the lower bound for $\delta_{\mathcal{NF}}^{(k)}(G).$
	\begin{proposition}\label{prop:lower-bound}
		Let $G=B_{n,m}$ with $\max\{n,m\}\ge 3$ and let $N=n+m$.
		Then $\delta_{\mathcal{NF}}^{(k)}(G)\not\cong G$ for $1\le k\le 2$.
		Moreover, for any $k$ such that $\dim(\delta_{\mathcal{NF}}^{(k)}(G))\ge 2$, then  $\delta_{\mathcal{NF}}^{(k)}(G)\not\cong G$.
	\end{proposition}
	
	\begin{proof}
		By Theorem~\ref{thm:delta1}, $\delta_{\mathcal{NF}}(G)$ is a bipartite graph with $nm-1$ edges,
		whereas $G$ has $\binom{n}{2}+\binom{m}{2}+1$ edges, so $\delta_{\mathcal{NF}}(G)\not\cong G$.
		By Theorem~\ref{thm:delta2}, $\delta_{\mathcal{NF}}^{(2)}(G)$ has facet $\A$ and $\B$, and hence has dimension at least $2$, if $\max\{n,m\}\ge 3$.
		Any complex of dimension greater or equal to $2$ cannot be isomorphic to the graph $G$, proving the claim.
	\end{proof}
	
	\medskip
	The following  result gives a structural two-block permanence phenomenon $\delta_{\mathcal{NF}}(\bigtriangleup_2).$
	\begin{proposition}\label{prop:two-block}
		Let $G=B_{n,m}$ with $\max\{n,m\}\ge 3$ and set $\bigtriangleup_2=\delta_{\mathcal{NF}}^{(2)}(G)$.
		Then every facet of $\delta_{\mathcal{NF}}(\bigtriangleup_2)=\delta_{\mathcal{NF}}^{(3)}(G)$ is of the form $(\A\setminus\{x\}) \cup (\B\setminus\{y\})$ 
		with $x\in\A$ and $y\in\B$, and at least one of $x,y$ is a bridge endpoint ($x=x_1$ or $y=y_1$).
	\end{proposition}
	
	\begin{proof}
		By Theorem~\ref{thm:delta2}, $\bigtriangleup_2$ has facets $\A$, $\B$, and $\{x_1,y_1\}$.
		A vertex cover of $\bigtriangleup_2$ must intersect all three facets.
		Minimality forces such a cover to have size $2$, as one vertex from $\A$ and one from $\B$ (otherwise one can delete redundant vertices
		and still intersect $\A$ and $\B$). The additional requirement to meet $\{x_1,y_1\}$ says that at least one chosen vertex is $x_1$ or $y_1$.
		Thus, the minimal covers are precisely the pairs $\{x_1,y\}$ with $y\in\B$ and $\{x,y_1\}$ with $x\in\A$.
		By Lemma~\ref{lem:facets-complements}, the facets of $\delta_{\mathcal{NF}}(\bigtriangleup_2)$ are the complements of such pairs, that is,
		$(\A\setminus\{x_1\})\cup(\B\setminus\{y\})$ or $(\A\setminus\{x\})\cup(\B\setminus\{y_1\})$.
	\end{proof}
	
	\medskip

		Proposition~\ref{prop:two-block} suggests that the whole orbit of a dumbbell may remain within a structured family of two-block complexes, whose facets are unions of large subsets of $\A$ and $\B$ with controlled deletions.
		A promising path to a full proof of Conjecture~\ref{conj:nf-dumbbell} is to take following into account:
		\begin{enumerate}[leftmargin=2em]
			\item define an explicit parameterized family $\{\bigtriangleup(t)\}$ that contains $\delta_{\mathcal{NF}}^{(k)}(B_{n,m})$ for all $k$, and
			\item show that $\delta_{\mathcal{NF}}$ acts as a predictable update rule on those parameters, producing period $n+m+2$.
		\end{enumerate}
		This would parallel the explicit orbit construction used by Hibi--Mahmood \cite{HibiMahmood2020} for $K_n\cup K_m$.
	
	\medskip
	Next, we illustrate results with the help of a dumbbell $B_{3,4}$ example.
	 	\begin{example}\label{ex:b34}
		Let $\A=\{x_1,x_2,x_3\}$ and $\B=\{y_1,y_2,y_3,y_4\}$.
		The dumbbell $G=B_{3,4}$ has clique edges inside $\A$, clique edges inside $\B$, and bridge $x_1y_1$.
		 A minimal vertex cover of $K(\A)$ must omit exactly one $x_i$. A similarly reason is true for $K(\B)$.
		So, the candidates are $C_{i,j}=(\A\setminus\{x_i\})\cup(\B\setminus\{y_j\})$ for $1\le i\le 3, 1\le j\le 4.$ The bridge $x_1y_1$ is uncovered only when $i=j=1$. Hence $\MIN(G)=\{C_{i,j}:(i,j)\neq(1,1)\}$.
		 By Lemma~\ref{lem:facets-complements}, the facets of $\delta_{\mathcal{NF}}(G)$ are complements, 
		$\V\setminus C_{i,j}=\{x_i,y_j\}$ for $(i,j)\neq(1,1)$.
		Hence, $\delta_{\mathcal{NF}}(G)$ is the bipartite graph with all $x_i y_j$ edges except $x_1y_1$.
		For $\delta_{\mathcal{NF}}^{(2)}(G)$, the minimal vertex covers of $H=\delta_{\mathcal{NF}}(G)=K_{3,4}\setminus\{x_1y_1\}$ are:
		$\A$, $\B$, and $(\A\setminus\{x_1\})\cup(\B\setminus\{y_1\})$ (Lemma~\ref{lem:mincov-bip-minus}).
		Complement give the facets $\B$, $\A$, and $\{x_1,y_1\}$, so we obtain
		 $\F(\delta_{\mathcal{NF}}^{(2)}(G))=\{\A, \B, \{x_1,y_1\}\}.$ 
	\end{example}
	
	\medskip
	 Graph characteristics can be used to define simple complexes in a variety of ways.
	González, Hoekstra, and Mendoza \cite{GonzalezHoekstraMendoza} examine the complex $D_{n,k}$, whose vertices are the $\binom{n}{2}$ possible edges on $[n]$, and whose simplices correspond to graphs $\sigma$ with domination number at least $k$.	Using discrete Morse theory, they show that $D_{n,n-2}$ has the same homotopy as a wedge of $2$-spheres for $n\ge 4$.
	
	From the NF perspective, a distinct dynamical structure is obtained, as by iterating $\delta_{\mathcal{NF}}$, we can create a periodic sequence of complexes from a single graph $G$ (as a $1$-complex). Both constructions highlight how rich simplicial-complex topology and combinatorics can be produced from graph-theoretic material.
	
	\begin{problem}[NF-iteration meets graph-complex topology]
		Is there a useful relationship between the NF-orbit of a graph $G$ and the homotopy type of
		graph complexes defined by monotone properties (e.g. domination, degree bounds, connectivity)?
		Does the NF-orbit, for instance, maintain or regulate any inherent Morse matchings or shellability properties?
	\end{problem}
	
	\section{NF-complexes and NF-numbers of the complete split graph}\label{section 4}
	  The complete split graph with clique number $n\geq 2$ and independence number $m\geq 1$ is defined as $S_{n,m} := K_n \vee \overline{K}_m$. Its vertex set is  a disjoint union $V = A\cup B,$ where $A=\{x_1,\dots,x_n\},$ and $B=\{y_1,\dots,y_m\}$. We consider $S_{n,m}$ to be a one-dimensional simplicial complex on the ground set $V$.	
	We begin with a detailed discussion of minimal vertex covers in the split graph $S_{n,m}$.	
	\begin{lemma}\label{lem:mincov-split}
		Let $G=S_{n,m}$ on $V=A\cup B$ as above.
		Then the minimal vertex covers of $G$ are precisely $A,$ and  $(A\setminus\{x_i\})\cup B $ for $ i=1,\dots,n.$
	\end{lemma}
	
	\begin{proof}
		Let $C\subseteq V$ be a vertex cover of $G$. We consider the following cases.\\		
		(1) Covering the clique on $A$:	Since $G[A]\cong K_n$, the set $C\cap A$ must contain at least $n-1$ vertices of $A$.	If $C$ omits two distinct vertices $x_i,x_j\in A$, then the clique edge $x_ix_j$	would have both endpoints outside $C$, contradicting that $C$ covers all edges. Thus,  $|C\cap A|\ge n-1$.\\
		(2) If $C\cap A=A$, then $C=A$ is minimal: If $A\subseteq C$, then every edge of $G$ is incident to some vertex of $A$,		so $A$ itself is a vertex cover. 		Moreover $A$ is  minimal, as for any $x_i\in A$, if we remove $x_i$ from $A$, then the edges $x_i y_j$
		(for any $j$) are no longer covered because $y_j\notin A$ and $x_i$ has been removed.
		Hence, no proper subset of $A$ is a vertex cover, so $A$ is a minimal vertex cover.\\
		(3) If $C\cap A=A\setminus\{x_i\}$, then $B\subseteq C$:		Now, assume that $|C\cap A|=n-1$, so $C\cap A=A\setminus\{x_i\}$ for a unique $i$.
		If any $y_j\in B$, then the edge $x_i y_j$ exists (all $A$--$B$ edges exist).
		Since $x_i\notin C$, to cover $x_i y_j$ we must have $y_j\in C$.
		As $j$ was arbitrary, so $B\subseteq C$ follows.\\
		 (4) Minimality forces no extra vertices:		By (3), any vertex cover with $|C\cap A|=n-1$ must contain
		$(A\setminus\{x_i\})\cup B$.
		This set is itself a vertex cover as clique edges in $A$ are covered since
		all except one vertex of $A$ are present, and every edge incident to $B$ is covered		because all vertices of $B$ are included.
		It is minimal by inclusion, as	removing any $y_j$ exposes the uncovered edge $x_i y_j$ (since $x_i\notin C$),
		and removing any $x_k\in A\setminus\{x_i\}$ exposes the clique edge $x_kx_i$.	Therefore, the only minimal covers of this type are exactly
		$(A\setminus\{x_i\})\cup B$. Now, combining (2) and (4) gives the given list.
	\end{proof}
	
	\medskip
	Applying Lemma~\ref{lem:mincov-split}, we have the following result, which presents the facets of $\delta_{NF}(\bigtriangleup)$, which 
	are the complements of minimal vertex covers of simplicial complex $\bigtriangleup $ on ground set $V$.
	\begin{proposition}\label{prop:NF1-split}
		Let $\bigtriangleup(G)$ be the graph complex of $G=S_{n,m}$.
		Then $\delta_{NF}(\bigtriangleup(G))$ has facets
		\[
		\mathcal{F}\bigl(\delta_{NF}(\bigtriangleup(G))\bigr)
		=\{\,B\,\} \cup \bigl\{\{x_1\},\dots,\{x_n\}\bigr\}.
		\]
	\end{proposition}
	
	\begin{proof}
		By Lemma~\ref{lem:mincov-split}, the minimal vertex covers are
		$A$ and $(A\setminus\{x_i\})\cup B$.
		Taking complements in $V=A\cup B$, we get  $V\setminus A = B,$ and $V\setminus\bigl((A\setminus\{x_i\})\cup B\bigr)=\{x_i\}.$ These complements are inclusion-maximal among themselves, so they are precisely
		the facets of $\delta_{NF}(\bigtriangleup(G))$.
	\end{proof}
	
	\medskip 
	A second iterate is described in the following lemma.	
	\begin{proposition}\label{prop:NF2-split}
		For $G=S_{n,m}$, the complex $\delta_{NF}^{(2)}(\bigtriangleup(G))$ has facets
		\[
		\mathcal{F}\bigl(\delta_{NF}^{(2)}(\bigtriangleup(G))\bigr)
		=\bigl\{\,B\setminus\{y_j\} : j=1,\dots,m\,\bigr\}.
		\]
	\end{proposition}
	
	\begin{proof}
		Let $\bigtriangleup_1=\delta_{NF}(\bigtriangleup(G))$.
		By Proposition~\ref{prop:NF1-split}, the facets of $\bigtriangleup_1$ are $B$ and the singletons
		$\{x_i\}$.
		Thus, a subset $C\subseteq V$ is a vertex cover of $\bigtriangleup_1$ if and only if it intersects $B$
		and also intersects every singleton facet $\{x_i\}$, that is, if and only if
		 $\{x_1,\dots,x_n\}\subseteq C$ and $C\cap B\neq\emptyset.$ Hence, the minimal vertex covers of $\bigtriangleup_1$ are exactly $A\cup\{y_j\}$,  for $j=1,\dots,m.$
		With complements in $V$, we obtain $V\setminus (A\cup\{y_j\}) = B\setminus\{y_j\},$ and these are maximal under inclusion among the obtained complements. Therefore, they are the facets of $\delta_{NF}(\bigtriangleup_1)=\delta_{NF}^{(2)}(\bigtriangleup(G))$.
	\end{proof}
	
	\medskip
	  	Let $\bigtriangleup(G)$ denote the edge complex of $G=S_{n,m}$. Write $\mathfrak{S}_A$ for the full symmetric group on the set $A$ (all permutations of the $x_i$'s),	and $\mathfrak{S}_B$ for the full symmetric group on $B$ (all permutations of the $y_j$'s). The direct product $\mathfrak{S}_A\times \mathfrak{S}_B$ 
	acts on $V$ by permuting vertices inside $A$ and inside $B$ independently.
	This action preserves adjacency in $S_{n,m}$ (it preserves clique edges in $A$, preserves
	the fact that there are no edges inside $B$, and preserves all cross edges), hence it is a subgroup
	of the automorphism group $\mathrm{Aut}(G)$. Consequently, it acts on $\bigtriangleup(G)$ and on every $\delta_{NF}^{(k)}(\bigtriangleup(G))$
	by relabeling vertices.
	
	The following lemma gives the equivariance of $\delta_{NF}$ under relabeling.
	\begin{lemma}\label{lem:equivariance}
		Let $\pi$ be a permutation of $V$ and let $\bigtriangleup$ be a simplicial complex on $V$.
		Then $\delta_{NF}(\pi(\bigtriangleup))=\pi(\delta_{NF}(\bigtriangleup))$. In particular, if $\bigtriangleup$ is invariant
		under a group of permutations $\mathcal{G}$, then so is $\delta_{NF}(\bigtriangleup)$ and every iterate
		$\delta_{NF}^{(k)}(\bigtriangleup)$.
	\end{lemma}
	
	\begin{proof}
		By definition, $\delta_{NF}(\bigtriangleup)=\bigtriangleup_{\mathcal{N}}(\IF(\bigtriangleup))$, where $\IF(\bigtriangleup)$ is the facet ideal
		and $\bigtriangleup_{\mathcal{N}}(\cdot)$ is the Stanley--Reisner complex. Relabeling vertices by $\pi$ sends each facet
		$F$ to $\pi(F)$ and hence sends each generator $\prod_{v\in F} v$ of $\IF(\bigtriangleup)$ to the generator
		$\prod_{v\in \pi(F)} v$ of $\IF(\pi(\bigtriangleup))$. Thus, $\IF(\pi(\bigtriangleup))=\pi(\IF(\bigtriangleup))$, and taking
		Stanley--Reisner complexes commutes with relabeling, so we have
		\[
		\delta_{NF}(\pi(\bigtriangleup))=\bigtriangleup_N(\IF(\pi(\bigtriangleup)))=\bigtriangleup_N(\pi(\IF(\bigtriangleup)))=\pi(\bigtriangleup_N(\IF(\bigtriangleup)))
		=\pi(\delta_{NF}(\bigtriangleup)).\qedhere
		\]
	\end{proof}
	
	\medskip
	For any subset $F\subseteq V$, define its \emph{type} by
	\[
	\mathrm{type}(F):=(|F\cap A|, |F\cap B|)\in\{0,1,\dots,n\}\times\{0,1,\dots,m\}.
	\]
	For integers $0\le i\le n$ and $0\le j\le m$, define the \emph{type-class}
	\[
	\mathcal{M}(i,j):=\{\,F\subseteq V : \mathrm{type}(F)=(i,j)\,\}.
	\]
	Note that $\mathfrak{S}_A\times\mathfrak{S}_B$ acts transitively on $\mathcal{M}(i,j)$. Any combination of $i$ vertices in $A$ and $j$ vertices in $B$ can be carried to any other by permuting them inside each part.
	The following lemma shows that invariant complexes have facet sets that are unions of type-classes.
	\begin{lemma}\label{lem:type-union}
		Let $\bigtriangleup$ be a simplicial complex on $V=A\cup B$ that is invariant under
		$\mathfrak{S}_A\times \mathfrak{S}_B$. If $F$ is a facet of $\bigtriangleup$ of type $(i,j)$,
		then \emph{every} set in $\mathcal{M}(i,j)$ is a facet of $\bigtriangleup$. Hence there exists
		a set $\alpha\subseteq[0,n]\times[0,m]$ such that $\mathcal{F}(\bigtriangleup)=\bigcup_{(i,j)\in\alpha}\mathcal{M}(i,j).$ 
	\end{lemma}
	
	\begin{proof}
		Let $F$ be a facet of type $(i,j)$. For any $F'\in \mathcal{M}(i,j)$, transitivity gives a
		$\pi\in\mathfrak{S}_A\times\mathfrak{S}_B$ with $\pi(F)=F'$. Since $\bigtriangleup$ is invariant,
		$\pi$ maps facets to facets, so $F'$ is a facet. Therefore, all of $\mathcal{M}(i,j)$ lies in
		$\mathcal{F}(\bigtriangleup)$. Taking $\alpha$ to be the set of types appearing among facets yields the
		displayed union.
	\end{proof}
	
	\medskip
	Put the product (componentwise) partial order on $[0,n]\times[0,m]$
	 $(i,j)\le (i',j')$ if and only if $i\le i'$ and $j\le j'.$ 
	A subset $\alpha\subseteq[0,n]\times[0,m]$ is an \emph{antichain}, if no two distinct elements of $\alpha$
	are comparable under this order. The following lemma shows that facet-type sets are antichains.
	
	\begin{lemma}\label{lem:alpha-antichain}
		Let $\bigtriangleup$ be invariant under $\mathfrak{S}_A\times\mathfrak{S}_B$ and write
		$\mathcal{F}(\bigtriangleup)=\bigcup_{(i,j)\in\alpha}\mathcal{M}(i,j)$ as in Lemma~\ref{lem:type-union}.
		Then $\alpha$ is an antichain.
	\end{lemma}
	
	\begin{proof}
		Assume for contradiction that $(i,j),(i',j')\in\alpha$ with $(i,j)<(i',j')$ componentwise, so
		$i\le i'$, $j\le j'$, and at least one inequality is strict. Choose any facet $F\in\mathcal{M}(i,j)$ and any set $F'\in\mathcal{M}(i',j')$. Then as $i\le i'$ and $j\le j'$, we can extend $F$ to some set $\widetilde{F}$ of type $(i',j')$
		by adding $(i'-i)$ vertices from $A\setminus F$ and $(j'-j)$ vertices from $B\setminus F$.
		Thus, $F\subsetneq \widetilde{F}$ and $\widetilde{F}\in\mathcal{M}(i',j')$. By Lemma~\ref{lem:type-union}, every set in $\mathcal{M}(i',j')$ is a facet, so in particular
		$\widetilde{F}$ is a facet. But then $F$ cannot be a facet (it is properly contained in the face
		$\widetilde{F}$), contradicting that $F$ was a facet. Hence, no such comparable pair exists and
		$\alpha$ is an antichain.
	\end{proof}
	
	\medskip
	\noindent\textbf{Covers in terms of types and rowmotion.}
	Let $P:=[0,n]\times[0,m]$ with the product order.
	Define the involution $\varsigma:P\to P$ by $\varsigma(a,b):=(n-a, m-b).$
	For an antichain $\psi\subseteq P$, let $I(\psi)\subseteq P$ be the order ideal it generates:
	\[
	I(\psi):=\{\,u\in P:\exists~ v\in\psi \text{ with } u\le v\,\}.
	\]
	Define \emph{rowmotion} on antichains by $\rho(\psi):=\min\bigl(P\setminus I(\psi)\bigr),$ 
	the set of minimal elements of the complement of the ideal. Next, we have as result showing that the vertex-cover types are complements of an ideal.	
	\begin{lemma}\label{lem:covers-as-ideal}
		Let $\bigtriangleup$ be invariant under $\mathfrak{S}_A\times\mathfrak{S}_B$ and suppose
		$\mathcal{F}(\bigtriangleup)=\bigcup_{(i,j)\in\alpha}\mathcal{M}(i,j)$ with $\alpha$ an antichain.
		Let $C\subseteq V$ and set $\mathrm{type}(C)=(a,b)$. Then $C$ is a vertex cover of $\bigtriangleup$ if and only if $(a,b)\notin I(\varsigma(\alpha)).$ 
		Moreover, the set of \emph{minimal} vertex cover types is exactly $\beta(\bigtriangleup)=\rho(\varsigma(\alpha)).$ 
	\end{lemma}
	
	\begin{proof}
		For a fixed $C$ of type $(a,b)$. The set $C$ fails to be a vertex cover if and only if there exists a facet
		$F$ disjoint from $C$. Since facets run over all sets in each type-class, such an $F$ exists
		with $\mathrm{type}(F)=(i,j)\in\alpha$ if and only if we can choose $i$ vertices of $A$ and $j$ vertices of $B$
		entirely from $V\setminus C$, that is, if and only if $i\le |A\setminus C|=n-a,$ and  $j\le |B\setminus C|=m-b.$ Equivalently, $(a,b)\le(n-i,m-j)=\varsigma(i,j)$. Thus, we have $C$ is  not a vertex cover if and only if thee exists $(i,j)\in\alpha$ with$(a,b)\le\varsigma(i,j)$ if and only if $(a,b)\in I(\varsigma(\alpha)).$
		This proves the first claim. For minimality, among all cover types, the minimal ones are precisely the minimal elements of
		$P\setminus I(\varsigma(\alpha))$ (if $(a,b)$ is not minimal there is a strictly smaller cover type,
		hence some cover $C'$ properly contained in $C$ exists). Thus, by definition of rowmotion, we have
		\[
		\min(P\setminus I(\varsigma(\alpha)))=\rho(\varsigma(\alpha)).
		\]
		So, the minimal vertex cover types form $\beta(\bigtriangleup)=\rho(\varsigma(\alpha))$.
	\end{proof}
	
	\begin{lemma}\label{lem:type-update}
		With hypotheses as above, if $\alpha(\bigtriangleup)$ is the facet-type antichain of $\bigtriangleup$, then
		the facet-type antichain of $\delta_{NF}(\bigtriangleup)$ is $\alpha(\delta_{NF}(\bigtriangleup))=\varsigma(\rho(\varsigma(\alpha(\bigtriangleup)))).$ 
	\end{lemma}
	
	\begin{proof}
		As the facets of $\delta_{NF}(\bigtriangleup)$ are complements of minimal vertex covers.
		At the level of types, complementation sends $(a,b)$ to $\varsigma(a,b)$. So, by Lemma~\ref{lem:covers-as-ideal},
		minimal cover types are $\rho(\varsigma(\alpha(\bigtriangleup)))$. Taking complements yields the required formula.
	\end{proof}
	
	\medskip
	
	In $P=[0,n]\times[0,m]$, every order ideal $J\subseteq P$ can be encoded by a lattice path from
	$(0,m+1)$ to $(n+1,0)$ using steps $E=(1,0)$ and $S=(0,-1)$. The path is the upper-right boundary
	of $J$ (Ferrers boundary). Equivalently, we obtain a word $w(J)\in\{E,S\}^{n+m+2}$ 	with $(n+1)$ letters $E$ and $(m+1)$ letters $S$.  The boundary word $w(J)$ of an order ideal $J\subseteq [0,n]\times[0,m]$ is the word in the alphabet $\{E,S\}$ obtained by reading the upper-right lattice-path boundary of $J$ from $(0,m+1)$ to $(n+1,0)$, where $E=(1,0)$ and $S=(0,-1)$. Thus $w(J)$ has length $n+m+2$, with $n+1$ letters $E$ and $m+1$ letters $S$.
	The following fact shows that rowmotion rotates the boundary word.
	\begin{lemma}\label{lem:rowmotion-rotation}
		Let $\psi$ be an antichain in $P$ and set $J=I(\psi)$.
		Then $I(\rho(\psi))$ has boundary word obtained from $w(J)$ by cyclic rotation by one letter $w(I(\rho(\psi)))=\mathrm{rot}(w(J)),$ 		where $\mathrm{rot}$ moves the first letter of the word to the end.
	\end{lemma}
	
	\begin{proof}
		By the above boundary-corner description, the minimal elements of $P\setminus J$ are precisely the \emph{inner corners} of the Ferrers boundary of $J$. A point $(i,j)\notin J$ is minimal outside $J$ if and only if $(i-1,j)\in J$ and $(i,j-1)\in J$ (when defined),
		which is exactly the local pattern of an $ES$ corner on the boundary word.
		
		Rowmotion replaces the generating antichain of $J$ by those inner corners, that is, it shifts the cut point		on the cyclic boundary word forward by one boundary step. Concretely, if one reads the boundary as a cyclic word,
		choosing a different starting point changes $w(J)$ by rotation. The update $J\mapsto I(\rho(\psi))$ advances
		the starting point by exactly one step, hence performs $\mathrm{rot}$. 
	\end{proof}
	
	\medskip
	The following result presents the NF-number of the complete split graph $S_{n,m}.$
	 \begin{theorem}\label{thm:NFnumber-split}
		Let $G=S_{n,m}$ be a complete split graph on $n+m$ vertices with $n\ge 2$, $m\ge 1$.
		Then the NF-number of the edge complex $\bigtriangleup(G)$ satisfies
		 $\mathcal{NF}(G)=n+m+2.$ 
	\end{theorem}
	
	\begin{proof}
		Let $\bigtriangleup_0:=\bigtriangleup(G)$ and $\bigtriangleup_{k}:=\delta_{NF}^{(k)}(\bigtriangleup_0)$, for $k\ge 1$. 	By Lemma~\ref{lem:equivariance}, each $\bigtriangleup_k$ is invariant under $\mathfrak{S}_A\times\mathfrak{S}_B$.
		Hence, by Lemmas~\ref{lem:type-union} and \ref{lem:alpha-antichain}, each $\mathcal{F}(\bigtriangleup_k)$ is a union
		of type-classes indexed by an antichain $\alpha_k\subseteq P=[0,n]\times[0,m]$.
		 If $m=1$, then $S_{n,1}\cong K_{n+1}$, and it is known (and easy) that $c(K_{n+1})=n+2=n+m+2$.
		So, we  assume that $m\ge 2$. For computing $\alpha_0$,  the facets of $\bigtriangleup_0$ are exactly the edges of $G$.
		These are:
		(i) all edges inside the clique $A$ (type $(2,0)$), and
		(ii) all cross edges between $A$ and $B$ (type $(1,1)$).
		Thus, we obtain $\alpha_0=\{(2,0),(1,1)\}.$ 
		Also, by Lemma~\ref{lem:type-update}, we have $\alpha_{k+1}=\varsigma(\rho(\varsigma(\alpha_k)))$f or $k\ge 0.$
		 Thus, the type evolution is conjugate to rowmotion on antichains of $P$ (conjugating by $\varsigma$ does not
		change orbit size, since $\varsigma$ is an involution). For computing the boundary word of $I(\varsigma(\alpha_0))$, we have
		 $\varsigma(\alpha_0)=\{(n-2,m),(n-1,m-1)\}.$ Therefore, $J_0:=I(\varsigma(\alpha_0))$ consists of all pairs $(i,j)\le(n-2,m)$ together with all $(i,j)\le(n-1,m-1)$.
		The Ferrers boundary path from $(0,m+1)$ to $(n+1,0)$ is  $E^{\,n-1}\,S\,E\,S^{\,m}\,E.$ 
		So, the boundary word is
		\begin{equation}\label{eq:initial-word-split}
			w(J_0)=E^{n-1}\,S\,E\,S^{m}\,E
			\quad\in\{E,S\}^{n+m+2}.
		\end{equation}
		Now, to prove that the orbit has full period $n+m+2$. Since $m\ge 2$, so the word \eqref{eq:initial-word-split} has a unique maximal run of $S$'s, namely $S^m$.
		Any nontrivial cyclic rotation moves this unique length-$m$ run to a different cyclic position, and hence
		cannot fix the word. Therefore, the rotation orbit of $w(J_0)$ has size exactly $n+m+2$. By Lemma~\ref{lem:rowmotion-rotation}, rowmotion acts as a one-step rotation on the boundary word,
		so the rowmotion orbit of $\varsigma(\alpha_0)$ has size $n+m+2$. Since $\alpha_{k+1}=\varsigma\rho\varsigma(\alpha_k)$,
		the orbit of $\alpha_0$ under $\delta_{NF}$ has the same size. Hence, we obtain
		 $\bigtriangleup_{n+m+2}\cong \bigtriangleup_0,$ and  $\bigtriangleup_k\not\cong \bigtriangleup_0$ for $1\le k<n+m+2,$ 		which is exactly $\mathcal{NF}(G)=n+m+2$.
	\end{proof}
	
	\medskip
	We will illustrate above facts by the following example. 
	\begin{example}\label{ex:S32}
		Let $G=S_{3,2} $ where \(A=\{x_1,x_2,x_3\}\) is a clique set and \(B=\{y_1,y_2\}\) is an independent set (see Figure \ref{fig:S32}). Then by Theorem~\ref{thm:NFnumber-split}, we obtain \(c(S_{3,2})=7\).
		
		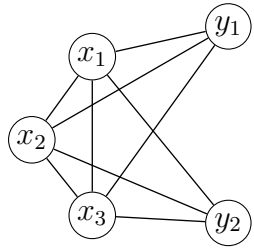
\begin{figure}[h]
			\centering
			\begin{tikzpicture}[
				scale=1,
				every node/.style={circle, draw, inner sep=1.5pt},
				xnode/.style={circle, draw, inner sep=1.5pt},
				ynode/.style={circle, draw, inner sep=1.5pt},
				ed/.style={line width=0.5pt}
				]
				\node[xnode] (x1) at (0,0.9) {$x_1$};
				\node[xnode] (x2) at (-0.8,-0.2) {$x_2$};
				\node[xnode] (x3) at (0,-1.2) {$x_3$};
				
				\node[ynode] (y1) at (1.8,1.3) {$y_1$};
				\node[ynode] (y2) at (1.8,-1.3) {$y_2$};
				
				\draw[ed] (x1)--(x2);
				\draw[ed] (x2)--(x3);
				\draw[ed] (x1)--(x3);
				
				\foreach \x in {x1,x2,x3}{
					\foreach \y in {y1,y2}{
						\draw[ed] (\x)--(\y);
					}
				}
			\end{tikzpicture}
			\caption{The complete split graph \(S_{3,2}=K_3\vee\overline{K}_2\).}
			\label{fig:S32}
		\end{figure}
		 \noindent Set \(\bigtriangleup_0:=\bigtriangleup(S_{3,2})\) and \(\bigtriangleup_k:=\delta_{NF}^{(k)}(\bigtriangleup_0)\).
		At each step, we compute \(\MIN(\bigtriangleup_k)\) and then use
		\[
		\mathcal{F}(\bigtriangleup_{k+1})
		=\mathcal{F}(\delta_{NF}(\bigtriangleup_k))
		=\{\,V\setminus C: C\in\MIN(\bigtriangleup_k)\,\}.
		\]
		 The facets of \(\bigtriangleup_0\) are the edges of the graph
		\[
		\mathcal{F}(\bigtriangleup_0)
		=
		\Bigl\{\{x_1,x_2\},\{x_1,x_3\},\{x_2,x_3\}\Bigr\}
		 \cup\
		\Bigl\{\{x_i,y_j\}: i\in\{1,2,3\}, j\in\{1,2\}\Bigr\}.
		\]
		For \(\bigtriangleup_1=\delta_{NF}(\bigtriangleup_0)\), we first determine \(\MIN(\bigtriangleup_0)\), that is, the minimal vertex covers of the graph \(S_{3,2}\).
		We will show that
		\[
		\MIN(\bigtriangleup_0)
		=
		\Bigl\{
		A, (A\setminus\{x_1\})\cup B, (A\setminus\{x_2\})\cup B, (A\setminus\{x_3\})\cup B
		\Bigr\}.
		\]
		It is clear that  \(A\) is a vertex cover since every edge has an endpoint in \(A\), and it is minimal, since removing \(x_i\)
			leaves the uncovered cross-edge \(\{x_i,y_1\}\).   If a vertex cover \(C\) omits some \(x_i\), then all edges \(\{x_i,y_1\},\{x_i,y_2\}\) force \(y_1,y_2\in C\), so \(B\subseteq C\). Moreover, to cover the clique edges inside \(A\), at most one \(x\)-vertex can be omitted, and  hence \(C\cap A=A\setminus\{x_i\}\). Thus, we obtain \(C=(A\setminus\{x_i\})\cup B\). Minimality is clear by checking that removing any \(y_j\) uncovers \(\{x_i,y_j\}\), and removing any \(x\in A\setminus\{x_i\}\) uncovers an edge inside \(K_3\).  Now, taking complements in \(V=A\cup B\), we obtain 
		\begin{align*}
		\mathcal{F}(\bigtriangleup_1)
		&=
		\{\,V\setminus A, V\setminus((A\setminus\{x_1\})\cup B), V\setminus((A\setminus\{x_2\})\cup B), V\setminus((A\setminus\{x_3\})\cup B)\,\}\\
		&=
		\bigl\{\{y_1,y_2\},\{x_1\},\{x_2\},\{x_3\}\bigr\}.
		\end{align*}
		For \(\bigtriangleup_2=\delta_{NF}(\bigtriangleup_1)\), from \(\mathcal{F}(\bigtriangleup_1)=\{\{y_1,y_2\},\{x_1\},\{x_2\},\{x_3\}\}\), a set \(C\subseteq V\) is a vertex cover of \(\bigtriangleup_1\) if and only if \(C\) intersects each singleton facet \(\{x_i\}\), so \(x_1,x_2,x_3\in C\) and   \(C\) intersects \(\{y_1,y_2\}\), hence \(C\cap B\neq\varnothing\). Therefore, the minimal vertex covers are exactly
		\[
		\MIN(\bigtriangleup_1)=\bigl\{\{x_1,x_2,x_3,y_1\}, \{x_1,x_2,x_3,y_2\}\bigr\}.
		\]
		With complements, we obtain $\mathcal{F}(\bigtriangleup_2)=\bigl\{\,\{y_2\}, \{y_1\}\bigr\}.$ 
		For  \(\bigtriangleup_3=\delta_{NF}(\bigtriangleup_2)\), as \(\mathcal{F}(\bigtriangleup_2)=\{\{y_1\},\{y_2\}\}\), so a vertex cover must meet both singletons, and the unique minimal vertex cover is \(B=\{y_1,y_2\}\).
		Hence,  $\MIN(\bigtriangleup_2)=\{B\}$ and it follows that $	\mathcal{F}(\bigtriangleup_3)=\{V\setminus B\}=\{A\}=\bigl\{\{x_1,x_2,x_3\}\bigr\}.$\\
		For \(\bigtriangleup_4=\delta_{NF}(\bigtriangleup_3)\). Since \(\bigtriangleup_3\) has the single facet \(A\), a set \(C\) is a vertex cover of \(\bigtriangleup_3\) if and only if \(C\cap A\neq\varnothing\).
		Thus, the minimal vertex covers are exactly \(\{x_1\},\{x_2\},\{x_3\}\), that is,
		 $\MIN(\bigtriangleup_3)=\bigl\{\{x_1\},\{x_2\},\{x_3\}\bigr\}.$ With complements, we obtain
		\[
		\mathcal{F}(\bigtriangleup_4)=\bigl\{V\setminus\{x_1\}, V\setminus\{x_2\}, V\setminus\{x_3\}\bigr\}
		=
		\bigl\{\{x_2,x_3,y_1,y_2\},\{x_1,x_3,y_1,y_2\},\{x_1,x_2,y_1,y_2\}\bigr\}.
		\]
		For \(\bigtriangleup_5=\delta_{NF}(\bigtriangleup_4)\), we write the three facets of \(\bigtriangleup_4\) as
	 $F_1=V\setminus\{x_1\}, F_2=V\setminus\{x_2\},$ and $ F_3=V\setminus\{x_3\}.$ A set \(C\) is a vertex cover of \(\bigtriangleup_4\) if and only if it meets each \(F_i\).
		There are two natural minimal ways to do this: we consider (i), all \(x\)'s and one \(y\). If \(x_1,x_2,x_3\in C\), then \(C\cap F_i\neq\varnothing\) for all \(i\), and minimality forces
			\(C\cap B\) to be a singleton (otherwise we could remove one \(y\)). Thus, we get \(\{x_1,x_2,x_3,y_1\}\) and \(\{x_1,x_2,x_3,y_2\}\). (ii) All \(y\)'s and one \(x\). If \(y_1,y_2\in C\), then \(C\cap F_i\neq\varnothing\) for all \(i\), and minimality forces
			\(C\cap A\) to be a singleton. Thus, we get \(\{x_i,y_1,y_2\}\), for \(i=1,2,3\).  So, these five sets are indeed minimal and that no other minimal cover exists (any cover must either contain both \(y\)'s or contain all three \(x\)'s to meet all \(F_i\) simultaneously). Hence, we have
		$$
		\operatorname{MinCov}(\Delta_4)
		=
		\bigl\{
		\{y_1\},\{y_2\},
		\{x_1,x_2\},\{x_1,x_3\},\{x_2,x_3\}
		\bigr\}.
		$$
		Thus, with complements, we obtain
		\[
		\mathcal{F}(\bigtriangleup_5)=
		\bigl\{\{x_1,x_2,x_3,y_1\},\{x_1,x_2,x_3,y_2\},\{x_1,y_1,y_2\},\{x_2,y_1,y_2\},\{x_3,y_1,y_2\}\bigr\}.
		\]
		The same family occurs because complements swap one \(y\) with the other \(y\) and similarly for \(x\)-choices.
		For \(\bigtriangleup_6=\delta_{NF}(\bigtriangleup_5)\), and \(\bigtriangleup_5\) has five facets, namely 
		 $G_1=A\cup\{y_1\}, G_2=A\cup\{y_2\},
		H_1=\{x_1\}\cup B, H_2=\{x_2\}\cup B,$ and $ H_3=\{x_3\}\cup B.$ 
		A minimal vertex cover must meet both \(G_1\) and \(G_2\), and must meet all three \(H_i\).
		There are exactly two minimal ways: (i) pick both \(y\)'s: \(B=\{y_1,y_2\}\) meets \(G_1,G_2\) and all \(H_i\), (ii)  pick all \(x\)'s: \(A\) meets \(G_1,G_2\) and all \(H_i\). Otherwise pick one \(x_i\) to meet \(H_i\) and one \(y_j\) to meet \(G_j\), as the six pairs \(\{x_i,y_j\}\) meet every facet and are minimal. Thus, we obtain
		\[
		\MIN(\bigtriangleup_5)=\{A, B\} \cup \bigl\{\{x_i,y_j\}: i=1,2,3, j=1,2\bigr\},
		\]
		and hence 
		\begin{align*}
			\mathcal{F}(\bigtriangleup_6)
			&=\{V\setminus A\} \cup \{V\setminus B\} \cup \bigl\{V\setminus\{x_i,y_j\}: i=1,2,3, j=1,2\bigr\}\\
			&=\bigl\{\{y_1,y_2\}\bigr\} \cup \bigl\{\{x_i,x_{i'},y_j\}:1\le i<i'\le 3, j=1,2\bigr\} \cup \bigl\{\{x_1,x_2,x_3\}\bigr\}.
		\end{align*}
		For \(\bigtriangleup_7=\delta_{NF}(\bigtriangleup_6)\), and from the description of \(\mathcal{F}(\bigtriangleup_6)\), a minimal vertex cover of \(\bigtriangleup_6\) is obtained by choosing a triple that hits:
		(i) the facet \(A\),
		(ii) the facet \(B\),
		and (iii) every triple \(\{x_i,x_{i'},y_j\}\).
		It is clear that the minimal vertex covers are exactly the nine triples
		\[
		\MIN(\bigtriangleup_6)=\Bigl\{
		\{x_i,x_{i'},y_j\}:1\le i<i'\le 3, j=1,2
		\Bigr\} \cup \Bigl\{\{x_i,y_1,y_2\}: i=1,2,3\Bigr\}.
		\]
		With complements, we obtain precisely the nine edges listed at the starting step as
		\[
		\mathcal{F}(\bigtriangleup_7)=\{\,V\setminus C: C\in \MIN(\bigtriangleup_6)\,\}=\mathcal{F}(\bigtriangleup_0).
		\]
		Hence the first return occurs at \(k=7\), and therefore $\mathcal{NF}(S_{3,2})=7.$ 
		 For convenience, the facet lists \(\mathcal{F}(\bigtriangleup_k)\) for \(0\le k\le 7\) are:
		\begin{align*}
			\mathcal{F}(\bigtriangleup_0)= &\{\{x_1,x_2\},\{x_1,x_3\},\{x_2,x_3\}\}
			\cup\{\{x_i,y_j\}: i\in\{1,2,3\}, j\in\{1,2\}\};\\
			\mathcal{F}(\bigtriangleup_1)= &\bigl\{\{y_1,y_2\},\{x_1\},\{x_2\},\{x_3\}\bigr\};~
			\mathcal{F}(\bigtriangleup_2)= \bigl\{\{y_1\},\{y_2\}\bigr\};\\	\mathcal{F}(\bigtriangleup_3)=&\bigl\{\{x_1,x_2,x_3\}\bigr\};~	\mathcal{F}(\bigtriangleup_4)= \bigl\{\{x_1,x_2,y_1,y_2\},\{x_1,x_3,y_1,y_2\},\{x_2,x_3,y_1,y_2\}\bigr\};\\[4pt]
			\mathcal{F}(\bigtriangleup_5)= &
			\bigl\{\{x_1,x_2,x_3,y_1\},\{x_1,x_2,x_3,y_2\},\{x_1,y_1,y_2\},\{x_2,y_1,y_2\},\{x_3,y_1,y_2\}\bigr\};\\[4pt]
			\mathcal{F}(\bigtriangleup_6)= &
			\bigl\{\{y_1,y_2\}\bigr\} \cup\
			\bigl\{\{x_i,x_{i'},y_j\}:1\le i<i'\le 3, j=1,2\bigr\} \cup\
			\bigl\{\{x_1,x_2,x_3\}\bigr\};\\[4pt]
			\mathcal{F}(\bigtriangleup_7)= &\mathcal{F}(\bigtriangleup_0).
		\end{align*}
	\end{example}
	 Here $n=3$, $m=2$, so the initial boundary word from \eqref{eq:initial-word-split} is
	 $w(J_0)=E^{2}SE S^{2}E = EES\,E\,SSE,$ which has length $7$ and hence full rotation period $7$.

		This example illustrates the mass transfer mechanism in NF-iteration for split graphs:
		\[
		\text{graph (edges)} \longrightarrow \{B\}\cup\{\text{singletons in }A\} \longrightarrow \{\text{singletons in }B\} \longrightarrow \{A\}.
		\]
		For larger $(n,m)$ the same phenomenon persists, but the intermediate facet sizes become richer. After two steps one typically sees
		the family $\{B\setminus\{y_j\}:j=1,\dots,m\}$, and subsequent iterates correspond to complements of minimal transversals of these uniform families.

		\section{Double stars and the first three \(\delta_{\mathcal{NF}}\)-iterates}\label{section 5}
		Fix integers \(p,q\ge 1\), consider  $\V = \{u,v\} \cup A \cup B, A=\{a_1,\dots,a_p\},$ and $ B=\{b_1,\dots,b_q\}.$ The \emph{double star} \(D_{p,q}\) is the tree on \(\V\) with edges $\{u,v\}, \{u,a_i\}, 1\le i\le p $and $ \{v,b_j\}, 1\le j\le q.$
		 Next, result gives the minimal vertex covers of \(D_{p,q}\).
		\begin{lemma}\label{lem:mincov-double-star}
			Let \(D_{p,q}\) be the double star. Then the family of minimal vertex covers of \(D_{p,q}\) is
			$\MIN(D_{p,q})=\bigl\{\,\{u,v\}, \{u\}\cup B, \{v\}\cup A\,\bigr\}.$ 
		\end{lemma}
		
		\begin{proof}
			Recall that a set \(C\subseteq V\) is a vertex cover if every edge of \(D_{p,q}\) has at least one endpoint in \(C\).
			We prove two that  each of the  sets $\bigl\{\,\{u,v\}, \{u\}\cup B, \{v\}\cup A\,\bigr\}$ is a minimal vertex cover, and  every minimal vertex cover is one of these three sets.  Indeed \(\{u,v\}\) is a vertex cover, since every edge of \(D_{p,q}\) is incident to either \(u\) or \(v\). The edges \(\{u,a_i\}\) are incident to \(u\), the edges \(\{v,b_j\}\) are incident to \(v\), and the edge \(\{u,v\}\) is incident to both. Hence \(\{u,v\}\) meets every edge.  The set \(\{u,v\}\) is minimal, as if we remove \(u\), then the edge \(\{u,a_1\}\) (or any \(\{u,a_i\}\)) is uncovered because neither endpoint lies in \(\{v\}\). If we remove \(v\), then the edge \(\{v,b_1\}\) is uncovered since neither endpoint lies in \(\{u\}\). Thus, no proper subset of \(\{u,v\}\) is a vertex cover.
			
			Another vertex cover set is \(\{u\}\cup B\), as it covers the bridge edge \(\{u,v\}\) because it contains \(u\).
			It covers every edge \(\{u,a_i\}\) because it contains \(u\). 	Finally, it covers every edge \(\{v,b_j\}\) since it contains \(b_j\in B\) (for each \(j\)).	So, \(\{u\}\cup B\) is a vertex cover. 
			If we remove \(u\), then \(\{u,v\}\) becomes uncovered as \(v\notin \{u\}\cup B\). 	If we remove some \(b_j\in B\), then the edge \(\{v,b_j\}\) becomes uncovered because \(v\notin \{u\}\cup (B\setminus\{b_j\})\). 	Thus, no vertex can be removed while preserving the vertex cover property, so \(\{u\}\cup B\) is minimal. 	The same argument (swapping the roles of \(u,A\) with \(v,B\)) shows \(\{v\}\cup A\) is a minimal vertex cover.
			
			 Let \(C\subseteq V\) be a minimal vertex cover of \(D_{p,q}\). Then we show that $C$ is one of the above listed sets. We consider the following cases.
			 
			\smallskip\noindent
			(1). If \(u\notin C\), then for each \(i\in\{1,\dots,p\}\), the edge \(\{u,a_i\}\) must be covered by \(a_i\), so \(a_i\in C\).
			Hence \(A\subseteq C\).
			Also the edge \(\{u,v\}\) must be covered, since \(u\notin C\), it follows that \(v\in C\).
			Therefore, we obtain $\{v\}\cup A  \subseteq C.$ 	But as proved above, \(\{v\}\cup A\) is itself a vertex cover. So, by minimality of \(C\) we must have equality  $C=\{v\}\cup A.$ 
			
			\smallskip\noindent
			(2). If \(v\notin C\), then by the symmetric argument, each edge \(\{v,b_j\}\) forces \(b_j\in C\), so \(B\subseteq C\),
			and the edge \(\{u,v\}\) forces \(u\in C\). Thus, \(\{u\}\cup B\subseteq C\), and minimality implies \(C=\{u\}\cup B\).
			
			\smallskip\noindent
			(3). If\(u\in C\) and \(v\in C\), then we show that minimality forces \(C=\{u,v\}\). 	Assume for contradiction that \(C\) contains some additional vertex \(z\in C\setminus\{u,v\}\).  Then \(z\in A\cup B\). Now, consider \(C':=C\setminus\{z\}\).
			We claim \(C'\) is still a vertex cover, and thereby we get contradiction to the minimality of \(C\).  Indeed, the only edges incident to \(z\) are:  if \(z=a_i\in A\), then the only incident edge is \(\{u,a_i\}\), which is still covered by \(u\in C'\); and  if \(z=b_j\in B\), then the only incident edge is \(\{v,b_j\}\), which is still covered by \(v\in C'\). 	All other edges are unchanged by removing \(z\). Hence \(C'\) remains a vertex cover, a contradiction.
			Therefore, no such \(z\) exists and \(C=\{u,v\}\). Thus, every minimal vertex cover \(C\) is exactly one of
			\(\{u,v\}\), \(\{u\}\cup B\), \(\{v\}\cup A\), completing the proof.
		\end{proof}
		
		\medskip
		The following result gives the first iteration of \(\bigtriangleup_1=\delta_{\mathcal{NF}}(\bigtriangleup_0)\) where $ \bigtriangleup_0:=D_{p,q}.$
		\begin{proposition}\label{prop:NF1-double-star}
			Let \(\bigtriangleup_0:=D_{p,q}\) be the double star on $V=\{u,v\}\cup A\cup B,$ 
			and let \(\bigtriangleup_1:=\delta_{\mathcal{NF}}(\bigtriangleup_0)\). Then
			 $\F(\bigtriangleup_1)=\bigl\{\,A\cup B, A\cup\{v\}, B\cup\{u\}\,\bigr\}.$ 
		\end{proposition}
		
		\begin{proof}
			Since by definition of the NF-complex,	the facets of \(\delta_{\mathcal{NF}}(\bigtriangleup_0)\) are precisely the complements in \(V\) of the minimal vertex covers of \(\bigtriangleup_0\)
			\[
			\F(\delta_{\mathcal{NF}}(\bigtriangleup_0))=\bigl\{\,V\setminus C : C\in \MIN(\bigtriangleup_0)\,\bigr\}.
			\]
			So, by Lemma~\ref{lem:mincov-double-star}, the minimal vertex covers of \(\bigtriangleup_0=D_{p,q}\) are exactly
			\[
			\MIN(\bigtriangleup_0)=\bigl\{\,\{u,v\}, \{u\}\cup B, \{v\}\cup A\,\bigr\}.
			\]
			Therefore, we compute the three complements in \(V\) as
			 $V\setminus\{u,v\}  =  A\cup B,$ 
			\[
			V\setminus(\{u\}\cup B)  =  (\{u,v\}\cup A\cup B)\setminus(\{u\}\cup B)
			 =  A\cup\{v\},
			\]
			\[
			V\setminus(\{v\}\cup A)  =  (\{u,v\}\cup A\cup B)\setminus(\{v\}\cup A)
			 =  B\cup\{u\}.
			\]
			Hence, we obtain
			\[
			\F(\bigtriangleup_1)=\F(\delta_{\mathcal{NF}}(\bigtriangleup_0))
			=\bigl\{\,A\cup B, A\cup\{v\}, B\cup\{u\}\,\bigr\}.
			\]
			Finally, each of these complements is inclusion-maximal among the faces of \(\bigtriangleup_1\). They are derived as complements of minimal vertex covers, so by construction, they constitute facets, and no other facets exist beyond these complements. This completes the proof.
		\end{proof}
		
		\medskip
		Next, we discuss the minimal vertex cover of \(\bigtriangleup_1\).
		 \begin{lemma}\label{lem:mincov-Delta1}
			Let \(p,q\ge 1\) and let \(\bigtriangleup_1=\delta_{\mathcal{NF}}(D_{p,q})\) be the first NF-iterate of the double star.
			Then the minimal vertex covers of \(\bigtriangleup_1\) are exactly the \(2\)-sets
			\[
			\MIN(\bigtriangleup_1)=
			\bigl\{\{u,a_i\}:1\le i\le p\bigr\} \cup\
			\bigl\{\{v,b_j\}:1\le j\le q\bigr\} \cup\
			\bigl\{\{a_i,b_j\}:1\le i\le p, 1\le j\le q\bigr\}.
			\]
		\end{lemma}
		
		\begin{proof}
			By Proposition~\ref{prop:NF1-double-star}, the facets of \(\bigtriangleup_1\) are $F_0:=A\cup B,  F_v:=A\cup\{v\}, $ and $ F_u:=B\cup\{u\}.$
			By definition, a set \(C\subseteq V\) is a vertex cover of \(\bigtriangleup_1\) if and only if $C\cap F_0\neq\varnothing,C\cap F_v\neq\varnothing, $ and $ C\cap F_u\neq\varnothing.$ 
			We prove the claim in three steps.
			
			\medskip\noindent
			\textbf{(1).} We need to show that every \(2\)-set listed in the statement is a vertex cover. 
			Let \(i\in\{1,\dots,p\}\) and \(j\in\{1,\dots,q\}\). Then we have the following facts. (i)   If \(C=\{u,a_i\}\), then \(C\cap F_0\supseteq\{a_i\}\neq\varnothing\), \(C\cap F_v\supseteq\{a_i\}\neq\varnothing\),
				and \(C\cap F_u\supseteq\{u\}\neq\varnothing\). Hence, \(C\) is a vertex cover. (ii)   If \(C=\{v,b_j\}\), then \(C\cap F_0\supseteq\{b_j\}\neq\varnothing\), \(C\cap F_v\supseteq\{v\}\neq\varnothing\),
				and \(C\cap F_u\supseteq\{b_j\}\neq\varnothing\). Hence, \(C\) is a vertex cover. (iii)   If \(C=\{a_i,b_j\}\), then \(C\cap F_0=\{a_i,b_j\}\neq\varnothing\), \(C\cap F_v\supseteq\{a_i\}\neq\varnothing\),
				and \(C\cap F_u\supseteq\{b_j\}\neq\varnothing\). Hence, \(C\) is a vertex cover. 
			
			\medskip\noindent
			\textbf{(2).} No singleton is a vertex cover as for each possible singleton \(\{x\}\subseteq V\), we can exhibit a facet disjoint from it, as \(
			\{u\}\cap F_v=\varnothing,		\{v\}\cap F_u=\varnothing,		\{a_i\}\cap F_u=\varnothing, \) and \(\{b_j\}\cap F_v=\varnothing.
			\)
			Thus, every vertex cover has size at least \(2\).
			
			\medskip\noindent
			\textbf{(3).} Every minimal vertex cover has size \(2\) and must be one of the listed \(2\)-sets.	Let \(C\) be a minimal vertex cover of \(\bigtriangleup_1\).
			By Step 2, \(|C|\ge 2\). We show \(|C|=2\).
			Since \(C\cap F_v\neq\varnothing\) and \(C\cap F_u\neq\varnothing\), we may choose $x\in C\cap F_v, a$ and $ y\in C\cap F_u.$	Then, the \(2\)-set \(C':=\{x,y\}\) satisfies \(C'\subseteq C\) and intersects both \(F_v\) and \(F_u\).
			We claim \(C'\) also intersects \(F_0\), hence is itself a vertex cover. The minimality of \(C\) then forces \(C=C'\).
			 To see \(C'\cap F_0\neq\varnothing\), note that \(x\in F_v=A\cup\{v\}\) and \(y\in F_u=B\cup\{u\}\).
			If \(x\in A\) or \(y\in B\), then \(x\in F_0\) or \(y\in F_0\), respectively, and we are done.
			So, the only remaining possibility is \(x=v\) and \(y=u\), but then $C'=\{u,v\} $ and $C'\cap F_0=\varnothing,$	contradicting that \(C\) (hence \(C'\subseteq C\)) must meet \(F_0=A\cup B\).			Therefore, \(C'\cap F_0\neq\varnothing\), and \(C'\) is a vertex cover. By minimality of \(C\), we conclude that  \(C=C'\). So, every minimal vertex cover has size exactly \(2\).
			
			Finally, since \(C=\{x,y\}\) with \(x\in F_v=A\cup\{v\}\) and \(y\in F_u=B\cup\{u\}\), and we have excluded the pair \(\{u,v\}\),
			the only possibilities are: \(x\in A\) and \(y=u\), giving \(C=\{u,a_i\}\) for some \(i\);  \(x=v\) and \(y\in B\), giving \(C=\{v,b_j\}\) for some \(j\), and   \(x\in A\) and \(y\in B\), giving \(C=\{a_i,b_j\}\) for some \(i,j\). These are exactly the \(2\)-sets listed in the statement. Each of the above listed \(2\)-set are minimal, since every singleton fails to be a cover (Step 2). Hence, no proper subset of a listed \(2\)-set can be a vertex cover. That completes the proof of  the lemma.
		\end{proof}
		
		\medskip
		Next, we have the second iteration \(\bigtriangleup_2=\delta_{\mathcal{NF}}(\bigtriangleup_1)\).
		\begin{proposition}
			Let \(\bigtriangleup_2:=\delta_{\mathcal{NF}}(\bigtriangleup_1)\). Then \(\F(\bigtriangleup_2)\) consists of the complements of the above $2$-covers
			\begin{align*}
			\F(\bigtriangleup_2)=&
			\bigl\{\V\setminus\{u,a_i\}: 1\le i\le p\bigr\}
			 \cup\
			\bigl\{\V\setminus\{v,b_j\}: 1\le j\le q\bigr\}\\
			& \cup\
			\bigl\{\V\setminus\{a_i,b_j\}: 1\le i\le p, 1\le j\le q\bigr\}.
			\end{align*}
		\end{proposition}
		
		\begin{proof}
			By Lemma \ref{lem:mincov-Delta1}, the proof is immediate from the defining rule \(\F(\delta_{\mathcal{NF}}(\bigtriangleup))=\{\V\setminus C: C\in\MIN(\bigtriangleup)\}\).
		\end{proof}
		
		\medskip
		 Let \(L:=A\cup\{v\}\) and \(R:=B\cup\{u\}\). The family of 2-sets in Lemma \ref{lem:mincov-Delta1} is precisely the edge set of the bipartite graph $H  =  K_{|L|,|R|}\setminus \{uv\},$ 	where \(u\in R\), \(v\in L\), and the unique missing cross-edge is \(\{u,v\}\).
			Then \(\bigtriangleup_2\) is the simplicial complex whose facets are \(\V\setminus e\) for edges \(e\in E(H)\). 
		
		The following lemma gives minimal vertex covers of \(\bigtriangleup_2\).
		\begin{lemma}\label{lem:mincov-Delta2}
			Let \(\bigtriangleup_2=\delta_{\mathcal{NF}}(\bigtriangleup_1)\) be the second NF-iterate of \(D_{p,q}\), and
			let \(L:=A\cup\{v\}\) and \(R:=B\cup\{u\}\). 
			Then 
			\[
			\MIN(\bigtriangleup_2)=\binom{A\cup\{v\}}{2} \cup \binom{B\cup\{u\}}{2} \cup \{\{u,v\}\}.
			\]
			Equivalently, \(\MIN(\bigtriangleup_2)\) is the edge set of the dumbbell graph consisting of a clique on \(A\cup\{v\}\),
			a clique on \(B\cup\{u\}\), and the bridge edge \(\{u,v\}\).
		\end{lemma}
		
		\begin{proof}
			By Lemma~\ref{lem:mincov-Delta1} and the definition of \(\delta_{\mathcal{NF}}\),
			the facets of \(\bigtriangleup_2=\delta_{\mathcal{NF}}(\bigtriangleup_1)\) are complements of the minimal vertex covers of \(\bigtriangleup_1\). Let \(H\) be the bipartite graph on the bipartition \(L\cup R\) with edge set $E(H)=\bigl(\{u\}\times A\bigr) \cup \bigl(A\times B\bigr) \cup \bigl(B\times \{v\}\bigr), $ or equivalently \(H\cong K_{p+1,q+1}\setminus\{uv\}\) (the unique missing cross-edge is \(\{u,v\}\)).
			Since \(\MIN(\bigtriangleup_1)\) is exactly the set of edges of \(H\), we may write
			\begin{equation}\label{eq:facets-Delta2}
				\F(\bigtriangleup_2)=\{ V\setminus e  :  e\in E(H) \}.
			\end{equation}
			We consider the following steps.\\
			\medskip\noindent
			\textbf{(1).} Here we reformulate the vertex cover condition. Let \(C\subseteq V\). Then \(C\) is a vertex cover of \(\bigtriangleup_2\) if and only if it intersects every facet of \(\bigtriangleup_2\). 	Using \eqref{eq:facets-Delta2}, this is equivalent to: $\forall ~e\in E(H),$   $C\cap (V\setminus e)\neq\varnothing.$ 
			For a fixed \(e\subseteq V\), the condition \(C\cap (V\setminus e)=\varnothing\) holds if and only if
			\(C\subseteq e\). Hence, the above condition is equivalent to
			\begin{equation}\label{eq:not-contained}
				(\forall ~e\in E(H))\qquad\text{and}\qquad C\not\subseteq e.
			\end{equation}
			
			\medskip\noindent
			\textbf{(2)}
			For fixed \(x\in V\). Since \(H\) has no isolated vertices (every vertex in \(L\cup R\) is incident to some edge of \(H\)),
			there exists an edge \(e\in E(H)\) with \(x\in e\). Then \(\{x\}\subseteq e\), so \eqref{eq:not-contained} fails.
			Thus, no singleton is a vertex cover of \(\bigtriangleup_2\).
			
			\medskip\noindent
			\textbf{(3).} Now, we characterize the \(2\)-element vertex covers.	Let \(C=\{x,y\}\) with \(x\neq y\). Then form \eqref{eq:not-contained}, we see that \(C\not\subseteq e\) for every \(e\in E(H)\).
			But each \(e\in E(H)\) is itself a \(2\)-set. Therefore, we have
			 $\{x,y\}\subseteq e$  if and only if $ \{x,y\}=e.$ 	Hence, we obtain
			\begin{equation}\label{eq:2cover-iff-nonedge}
				\{x,y\}\text{ is a vertex cover of }\bigtriangleup_2
				\quad\Longleftrightarrow\quad
				\{x,y\}\notin E(H).
			\end{equation}
			It follows that, a \(2\)-set is a vertex cover of \(\bigtriangleup_2\) if and only if it is a non-edge of \(H\).
			
			Since \(H\) is bipartite on \(L\cup R\), it has  no edges inside \(L\) and no edges inside \(R\).
			Therefore every \(2\)-subset of \(L\) and every \(2\)-subset of \(R\) is a non-edge of \(H\),
			and hence  a vertex cover of \(\bigtriangleup_2\) (by \eqref{eq:2cover-iff-nonedge}).
			Moreover, among cross pairs \(L\times R\), the only non-edge is \(\{u,v\}\) (because \(H\cong K_{p+1,q+1}\setminus\{uv\}\)).
			Thus, the family of \(2\)-element vertex covers of \(\bigtriangleup_2\) is exactly
			\[
			\binom{L}{2} \cup \binom{R}{2} \cup \{\{u,v\}\}
			=
			\binom{A\cup\{v\}}{2} \cup \binom{B\cup\{u\}}{2} \cup \{\{u,v\}\}.
			\]
			By step 2, no singleton is a vertex cover, so every \(2\)-element vertex cover is automatically minimal.
			Thus, the family displayed in step 3 is contained in \(\MIN(\bigtriangleup_2)\). Lastly, there exists no vertex cover of size \(|C|\ge 3\) is minimal. 
			Let \(C\subseteq V\) with \(|C|\ge 3\). Since \(V=L\cup R\) is a bipartition, at least two vertices of \(C\) lie in the same part,
			say \(x,y\in C\cap L\) (the case \(C\cap R\) is analogous). Then \(\{x,y\}\in \binom{L}{2}\), so by Step 3
			the \(2\)-set \(\{x,y\}\) is a vertex cover of \(\bigtriangleup_2\).
			Hence, \(C\) is not minimal (it contains a smaller vertex cover). Combining the above steps, we conclude that the \emph{minimal} vertex covers of \(\bigtriangleup_2\) are exactly
			\[
			\MIN(\bigtriangleup_2)=\binom{A\cup\{v\}}{2} \cup \binom{B\cup\{u\}}{2} \cup \{\{u,v\}\}.
			\]
			Thus, for the dumbbell graph, the above family is precisely the edge set of two cliques
			(on \(A\cup\{v\}\) and on \(B\cup\{u\}\)) together with the bridge \(\{u,v\}\).
		\end{proof}
		
		\medskip
		From the above proposition and Lemma \ref{lem:facets-complements}, we have the following result which gives the third iteration  \(\bigtriangleup_3=\delta_{\mathcal{NF}}(\bigtriangleup_2)\).
		\begin{proposition}
			Let \(\bigtriangleup_3:=\delta_{\mathcal{NF}}(\bigtriangleup_2)\). Then
			\[
			\F(\bigtriangleup_3)=
			\Bigl\{\V\setminus e : e\in \binom{A\cup\{v\}}{2}\Bigr\}
			 \cup\
			\Bigl\{\V\setminus e : e\in \binom{B\cup\{u\}}{2}\Bigr\}
			 \cup\
			\{\V\setminus\{u,v\}\}.
			\]
		\end{proposition}
		
		\begin{proof}
			Apply the defining facet rule to the description of \(\MIN(\bigtriangleup_2)\) from Lemma 2.5.
		\end{proof}
		
		\medskip
		 The following result gives the NF-number of a double star.
		 \begin{theorem}\label{thm:NFnumber-double-star}
			Let \(D_{p,q}\) be the double star  on $V=\{u,v\}\cup A\cup B,$
			where $A=\{a_1,\dots,a_p\},$ and $B=\{b_1,\dots,b_q\}.$
			Then the following hold.
			\begin{enumerate}
				\item If \((p,q)=(1,1)\), then \(D_{1,1}\cong P_4\) and \(\mathcal{NF}(D_{1,1})=1\).
				\item If \((p,q)\neq (1,1)\), then $\mathcal{NF}(D_{p,q})=p+q+4.$ 
			\end{enumerate}
		\end{theorem}
		
		\begin{proof}
			Let \(\bigtriangleup_0:=\bigtriangleup(D_{p,q})\) be the edge complex of \(D_{p,q}\), and \(\bigtriangleup_k:=\delta_{\mathcal{NF}}^{(k)}(\bigtriangleup_0)\).
			We use the facet--cover correspondence
			\begin{equation}\label{eq:facets-complements-double-star}
				\F(\delta_{\mathcal{NF}}(\bigtriangleup))=\{\,V\setminus C: C\in \MIN(\bigtriangleup)\,\},
			\end{equation}
			valid for every simplicial complex \(\bigtriangleup\).
			 If \((p,q)=(1,1)\), then \(D_{1,1}\) is the path $P_{4}:$ \(u-a_1-u-v-b_1\) on four vertices. 	It is known that \(\mathcal{NF}(P_4)=1\), that is, \(\delta_{\mathcal{NF}}(P_4)\cong P_4\) \cite{HibiMahmood2020}.
			 Next, we assume that \((p,q)\neq(1,1)\), and in particular, \(p+q\ge 3\). We prove that the period of the NF-iteration is exactly \(p+q+4\). We consider this fact in following steps.
			
			\smallskip\noindent
			\textbf{(1).} Here, we reduce the iteration to a symmetric type dynamics. 	Let $L:=A\cup\{v\},$ and $R:=B\cup\{u\},$
			so \(|L|=p+1\) and \(|R|=q+1\), and \(V=L\cup R\).
			Let \(\mathfrak{S}_L\times \mathfrak{S}_R\) act on \(V\) by permuting vertices inside \(L\) and inside \(R\).
			Since the description of \(\delta_{\mathcal{NF}}\) uses only vertex covers and complements, it commutes with relabeling. Hence every \(\bigtriangleup_k\) is invariant under \(\mathfrak{S}_L\times \mathfrak{S}_R\).
			Consequently, if a set \(F\subseteq V\) is a facet of \(\bigtriangleup_k\), then every set of the same cardinalities $(|F\cap L|, |F\cap R|)$ is also a facet.
			Define the \emph{type} of a set \(F\subseteq V\) by
			\[
			\mathrm{type}(F):=(|F\cap L|, |F\cap R|)\in [0,p+1]\times[0,q+1].
			\]
			For each type \((i,j)\) let \(\mathcal{M}(i,j)\) be the family of all sets \(F\subseteq V\) with \( \mathrm{type}(F)=(i,j)\).
			Then there exists an antichain \(\alpha_k\subseteq [0,p+1]\times[0,q+1]\) such that
			\begin{equation}\label{eq:facets-type-union}
				\F(\bigtriangleup_k)=\bigcup_{(i,j)\in \alpha_k}\mathcal{M}(i,j).
			\end{equation}
			Here \(\alpha_k\) is an antichain because if \((i,j)<(i',j')\) componentwise, then any face of type \((i,j)\)
			is properly contained in some face of type \((i',j')\), so it cannot be a facet.
			
			\smallskip\noindent
			\textbf{(2).} Here we identify the type update rule as rowmotion (up to an involution).	Let \(P:=[0,p+1]\times[0,q+1]\) with the componentwise order, and define the involution $\varsigma:P\to P$ by
			 $\varsigma(a,b)=(p+1-a, q+1-b).$ Given an antichain \(\psi\subseteq P\), let \(I(\psi)\subseteq P\) be the order ideal it generates and define rowmotion $\rho(\psi):=\min\bigl(P\setminus I(\psi)\bigr).$ We claim that the NF-step on facet-types is given by
			\begin{equation}\label{eq:NF-rowmotion}
				\alpha_{k+1}=\varsigma\circ \rho\circ \varsigma(\alpha_k)\qquad\text{for}\qquad (k\ge 0).
			\end{equation}
			Indeed, fix \(k\) and suppose \(\F(\bigtriangleup_k)\) has the form \eqref{eq:facets-type-union}.
			A set \(C\subseteq V\) of type \( \mathrm{type}(C)=(a,b)\) fails to be a vertex cover of \(\bigtriangleup_k\)
			if and only if there exists a facet \(F\) disjoint from \(C\). Since facets occur in full type-classes,
			there exists such an \(F\) of type \((i,j)\in\alpha_k\) if and only if one can choose \(i\) vertices from \(L\setminus C\) and
			\(j\) vertices from \(R\setminus C\), that is, if and only if
			 $i\le |L\setminus C|=(p+1)-a,$ and  $j\le |R\setminus C|=(q+1)-b.$ 
			Equivalently, we have \((a,b)\le \varsigma(i,j)\).
			Thus, \(C\) is a vertex cover if and only if \((a,b)\notin I(\varsigma(\alpha_k))\), and minimal vertex covers correspond exactly to the
			minimal elements of \(P\setminus I(\varsigma(\alpha_k))\), that is, to \(\rho(\varsigma(\alpha_k))\).
			Taking complements of minimal vertex covers (by \eqref{eq:facets-complements-double-star}), 
			apply \(\varsigma\) to these types, we obtain \eqref{eq:NF-rowmotion}.
			
			\smallskip\noindent
			\textbf{(3).} Here we verify the rowmotion on a rectangle is cyclic rotation on a boundary word.	It is known that order ideals in \(P=[0,p+1]\times[0,q+1]\) are in bijection with lattice paths from	\((0,q+2)\) to \((p+2,0)\) using steps \(E=(1,0)\) and \(S=(0,-1)\).	Equivalently, each ideal is encoded by a word of length $(p+2)+(q+2)=p+q+4$ 
			in the alphabet \(\{E,S\}\) having exactly \(p+2\) letters \(E\) and \(q+2\) letters \(S\).
			Under this encoding, rowmotion \(\rho\) acts by cyclic rotation of this boundary word by one position.
			Therefore, every rowmotion orbit has period dividing \(p+q+4\). Since \(\varsigma\) is an involution, \(\varsigma\rho\varsigma\) has the same orbit sizes as \(\rho\).
			Thus, by \eqref{eq:NF-rowmotion}, the type-orbit \(\alpha_0,\alpha_1,\dots\) has period dividing \(p+q+4\),
			hence \(\bigtriangleup_{p+q+4}\cong \bigtriangleup_0\).
			
			\smallskip\noindent
			\textbf{(4).} Here we show that the period is exactly \(p+q+4\). It remains to rule out an earlier return \(\bigtriangleup_t\cong \bigtriangleup_0\) for \(1\le t<p+q+4\). By steps 1--3, it suffices to show that the boundary word associated to the initial type-data has full rotation period. We now compute the initial boundary word from the explicit description of the first iterates.
			From Lemma~\ref{lem:mincov-Delta1} and Proposition~\ref{prop:NF1-double-star}, the second iterate \(\bigtriangleup_2\) has facets
			\(\{V\setminus e: e\in E(H)\}\), where \(H\cong K_{p+1,q+1}\setminus\{uv\}\) is bipartite on \(L\cup R\).
			By Lemma~\ref{lem:mincov-Delta2}, the minimal vertex covers of \(\bigtriangleup_2\) are exactly the 2-sets
			\[
			\binom{L}{2} \cup \binom{R}{2} \cup \{\{u,v\}\}.
			\]
			Applying \eqref{eq:facets-complements-double-star} once more, we obtain that \(\bigtriangleup_3\) has facets equal to complements
			of these 2-sets. In particular, \(\F(\bigtriangleup_3)\) consists of all subsets of \(V\) having type
			 $(p-1,q+1)$ or $(p+1,q-1)$ or $(p,q),$ together with certain subtypes forced by inclusions. Therefore, the facet-type antichain \(\alpha_3\) contains $(p-1,q+1)$  and $(p+1,q-1),$ and the ideal \(I(\varsigma(\alpha_3))\) has a boundary word with a unique maximal run of \(S\)'s (since one side is deficient by 2),
			which implies the rotation orbit has full length \(p+q+4\). More directly (and standard in the rectangle-rowmotion literature), the boundary word corresponding to the initial data for a
			double star can be chosen in the explicit form $w_0  =  E^{\,p+1}\,S\,E\,S^{\,q+1}\,E$ of length $p+q+4$.
			When \((p,q)\neq (1,1)\), at least one of \(p+1\ge 3\) or \(q+1\ge 3\) holds, hence \(w_0\) has a unique longest run among the
			\(E\)-blocks or among the \(S\)-blocks. Any nontrivial cyclic rotation moves that unique longest run to a different position,
			so no rotation by \(t\) with \(1\le t<p+q+4\) fixes \(w_0\). Hence the rotation orbit of \(w_0\) has size exactly \(p+q+4\).
			Therefore, the NF-type orbit has size \(p+q+4\), and consequently
			\[
			\bigtriangleup_t\cong\bigtriangleup_0  \Longleftrightarrow t\equiv 0\pmod{p+q+4}.
			\]
			In particular, \(\mathcal{NF}(D_{p,q})=p+q+4\), completing the proof of (b).
		\end{proof}
\medskip

For illustration, we have the following example \(D_{2,1}\) in Figure \ref{double star}.
		
		\begin{figure}[H]
			\centering
			\begin{tikzpicture}[scale=1.0, every node/.style={circle,draw,inner sep=1.5pt}]
				\node (u) at (0,0) {$u$};
				\node (v) at (2,0) {$v$};
				\node (a1) at (-1,0.8) {$a_1$};
				\node (a2) at (-1,-0.8) {$a_2$};
				\node (b1) at (3,0) {$b_1$};
				
				\draw (u) -- (v);
				\draw (u) -- (a1);
				\draw (u) -- (a2);
				\draw (v) -- (b1);
			\end{tikzpicture}
			\caption{The double star \(D_{2,1}\).}
			\label{double star}
		\end{figure}
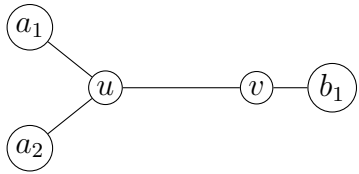
		
		 \begin{example}\label{ex:D21}
			Consider the graph \(D_{2,1}\) as shown in Figure \ref{double star}. Thus $V=\{u,v,a_1,a_2,b_1\},$ and $E(D_{2,1})=\bigl\{\{u,v\},\{u,a_1\},\{u,a_2\},\{v,b_1\}\bigr\},$
			and we set \(\bigtriangleup_0:=\bigtriangleup(D_{2,1})\) and \(\bigtriangleup_k:=\delta_{\mathcal{NF}}^{(k)}(\bigtriangleup_0)\).
			 By definition, \(\F(\bigtriangleup_0)=E(D_{2,1})\), that is,
			\[
			\F(\bigtriangleup_0)=\bigl\{\{u,v\},\{u,a_1\},\{u,a_2\},\{v,b_1\}\bigr\}.
			\]
			Let  (\(\bigtriangleup_1=\delta_{\mathcal{NF}}(\bigtriangleup_0)\)). Then, by Lemma~\ref{lem:mincov-double-star} (minimal covers of \(D_{p,q}\)) and the definition of \(\delta_{\mathcal{NF}}\), we have
			\[
			\MIN(\bigtriangleup_0)=\bigl\{\{u,v\},\{u\}\cup\{b_1\},\{v\}\cup\{a_1,a_2\}\bigr\}.
			\]
			Thus, with complements in \(V\),  the facets of \(\bigtriangleup_1\) are
			\[
			\F(\bigtriangleup_1)=
			\bigl\{V\setminus\{u,v\}, V\setminus\{u,b_1\}, V\setminus\{v,a_1,a_2\}\bigr\}
			=
			\bigl\{\{a_1,a_2,b_1\},\{a_1,a_2,v\},\{u,b_1\}\bigr\}.
			\]
			Next, for \(\bigtriangleup_2=\delta_{\mathcal{NF}}(\bigtriangleup_1)\), Lemma~\ref{lem:mincov-Delta1} (minimal covers of \(\bigtriangleup_1\)) gives
			\[
			\MIN(\bigtriangleup_1)=\{\{u,a_1\},\{u,a_2\},\{v,b_1\},\{a_1,b_1\},\{a_2,b_1\}\}.
			\]
			Thus, we have
			\begin{align*}
				\F(\bigtriangleup_2)
				&=\{\,V\setminus C: C\in\MIN(\bigtriangleup_1)\,\}\\
				&=\bigl\{\{v,a_2,b_1\},\{v,a_1,b_1\},\{u,a_1,a_2\},\{u,v,a_2\},\{u,v,a_1\}\bigr\}.
			\end{align*}
			Let \(\bigtriangleup_3=\delta_{\mathcal{NF}}(\bigtriangleup_2)\). Then we note that  \(L=A\cup\{v\}=\{a_1,a_2,v\}\) and \(R=B\cup\{u\}=\{b_1,u\}\), so by Lemma~\ref{lem:mincov-Delta2} (minimal covers of \(\bigtriangleup_2\)), we have
			\[
			\MIN(\bigtriangleup_2)=\binom{L}{2} \cup \binom{R}{2} \cup \{\{u,v\}\}
			=\{\{a_1,a_2\},\{a_1,v\},\{a_2,v\},\{u,b_1\},\{u,v\}\}.
			\]
			With complements, we obtain
			\begin{align*}
				\F(\bigtriangleup_3)
				&=\{\,V\setminus C: C\in\MIN(\bigtriangleup_2)\,\}\\
				&=\bigl\{\{u,v,b_1\},\{u,a_2,b_1\},\{u,a_1,b_1\},\{a_1,a_2,v\},\{a_1,a_2,b_1\}\bigr\}.
			\end{align*}
			 From the facet list of \(\bigtriangleup_3\), the minimal vertex covers are
			\[
			\MIN(\bigtriangleup_3)=\{\{a_1,b_1\},\{a_2,b_1\},\{a_1,u\},\{a_2,u\},\{b_1,v\},\{a_1,a_2,v\}\}.
			\]
			Each enumerated set satisfies all five features of \(\bigtriangleup_3\), as no singleton fulfills all facets, and any larger cover includes one of these sets.
			Hence, we obtain
			\begin{align*}
				\F(\bigtriangleup_4)
				&=\{\,V\setminus C: C\in\MIN(\bigtriangleup_3)\,\}\\
				&=\bigl\{\{u,v,a_2\},\{u,v,a_1\},\{v,a_2,b_1\},\{v,a_1,b_1\},\{u,a_1,a_2\},\{u,b_1\}\bigr\}.
			\end{align*}
			 From the facet list of \(\bigtriangleup_4\), the minimal vertex covers are
			\[
			\MIN(\bigtriangleup_4)=\{\{u,b_1\},\{u,v\},\{a_1,a_2,b_1\},\{a_1,a_2,u\},\{a_1,b_1,v\},\{a_2,b_1,v\}\}.
			\]
			Therefore, we obtain
			\begin{align*}
				\F(\bigtriangleup_5)
				&=\{\,V\setminus C: C\in\MIN(\bigtriangleup_4)\,\}\\
				&=\bigl\{\{a_1,a_2,v\},\{a_1,a_2,b_1\},\{u,v\},\{v,b_1\},\{u,a_2\},\{u,a_1\}\bigr\}.
			\end{align*}
			For \(\bigtriangleup_6=\delta_{\mathcal{NF}}(\bigtriangleup_5)\), and from \(\F(\bigtriangleup_5)\), the minimal vertex covers are the following \(3\)-sets
			\[
			\MIN(\bigtriangleup_5)=\{\{a_1,a_2,v\},\{a_1,b_1,u\},\{a_1,u,v\},\{a_2,b_1,u\},\{a_2,u,v\},\{b_1,u,v\}\}.
			\]
			Hence, we have
			\begin{align*}
				\F(\bigtriangleup_6)
				&=\{\,V\setminus C: C\in\MIN(\bigtriangleup_5)\,\}\\
				&=\bigl\{\{u,b_1\},\{a_2,v\},\{a_2,b_1\},\{a_1,v\},\{a_1,b_1\},\{a_1,a_2\}\bigr\}.
			\end{align*}
			Finally, for \(\bigtriangleup_7=\delta_{\mathcal{NF}}(\bigtriangleup_6)\), and from \(\F(\bigtriangleup_6)\), the minimal vertex covers are
			\[
			\MIN(\bigtriangleup_6)=\{\{a_1,a_2,b_1\},\{a_1,a_2,u\},\{a_1,b_1,v\},\{a_2,b_1,v\}\},
			\]
			and we obtain
			\begin{align*}
				\F(\bigtriangleup_7)
				&=\{\,V\setminus C: C\in\MIN(\bigtriangleup_6)\,\} =\bigl\{\{u,v\},\{v,b_1\},\{u,a_2\},\{u,a_1\}\bigr\}
				=\F(\bigtriangleup_0).
			\end{align*}
			Thus, the first return occurs at \(k=7\), so \(\mathcal{NF}(D_{2,1})=7=p+q+4\).
		\end{example}

	\section{Conclusion}
	
	We gave detailed, self-contained foundations for NF-complexes and NF-numbers, and proved explicit new formulas for
	$\delta_{\mathcal{NF}}(B_{n,m})$ and $\delta_{\mathcal{NF}}^{(2)}(B_{n,m})$. These computations show that dumbbells map first to an almost complete	bipartite graph, then to a rigid three-facet complex. This structure strongly suggests the sharp conjecture
	$\mathcal{NF}(B_{n,m})=n+m+2$ (outside the smallest exceptional case), and we verified it computationally for $2\le n,m\le 5$.
	A complete proof appears to require a full parameterization of the orbit, analogous to the orbit analysis for $K_n\cup K_m$. Theorem \ref{thm:NFnumber-split} gives the NF-number of complete split graph, and Theorem \ref{thm:NFnumber-double-star} gives the NF-number of double star. However, NF-number of general graphs remains yet to be discussed, and other homological invariants of NF-complexes can be investigated in future.  It would be valuable to complete the validation of the dumbbell conjecture.  Second, one may ask which graph families admit a rowmotion model similar to the one used here for complete split graphs and double stars. Natural candidates may include threshold graphs, complete multipartite graphs, block graphs, chordal graphs, and trees of small diameter. Such a classification would help explain when NF-periods are governed by simple rotation phenomena. Third, the NF-number of general trees remains largely open. Double stars provide the first nontrivial step beyond paths and stars, but more complicated trees may have orbits with richer intermediate complexes. A systematic study of how branching affects the NF-orbit could lead to new formulas or useful bounds. Fourth, the labelled and unlabelled versions of the NF-number deserve further attention. In small examples, especially those related to $P_4$, the first return may depend on whether one works on a fixed labelled vertex set or identifies complexes up to isomorphism at each stage. Clarifying this distinction may help reconcile different conventions in the literature. Fifth, it would be interesting to study algebraic and topological invariants along the NF-orbit. For example, one may ask how dimension, $f$-vectors, Betti numbers, shellability, Cohen--Macaulayness, and homotopy type change under repeated application of $\delta_{\mathcal{NF}}$. Since the NF-operator is defined through the facet ideal and the Stanley--Reisner complex, these questions naturally connect combinatorial dynamics with combinatorial commutative algebra.  Finally, there is room for computational development. An efficient implementation of NF-iteration, based on minimal transversals and vertex-cover enumeration, could be used to build a database of NF-numbers for small graphs and simplicial complexes. Such data may reveal new patterns, suggest additional conjectures, and provide test cases for future theoretical work.  Overall, the results of this paper show that even for familiar graph families, the NF-operator produces rich and structured periodic behaviour. The interaction between vertex covers, simplicial complexes, and rowmotion appears to be a promising framework for further study.

	\section*{Declarations}
	\noindent \textbf{Data Availability:}	There is no data associated with this article.
	
	\medskip
	
	\noindent \textbf{Funding:} The authors did not receive support from any organization for the submitted work.
	
	\medskip
	
	\noindent \textbf{Conflict of interest:} The authors have no competing interests to declare that are relevant to the content of this article.
	
	\medskip
	
	\noindent\textbf{Note:} For any comments and suggestions regarding this article, please feel free to contact at \href{mailto:bilalahmadrr@gmail.com}{bilalahmadrr@gmail.com}.


\begin{thebibliography}{99}
		
		\bibitem{bilalhafiz} H. M. Bilal, S. Ahmed, H. Mahmood, and M. Binyamin,  The NF-number of two complete graphs joined by a common vertex, \textit{Proc. Bulgarian Acad. Sci.} \textbf{75}(5) (2022) 640--648.
		
		\bibitem{bilalhafiz1} H. M. Bilal, S. Ahmad, H. Mahmood, M. A. Binyamin, Expected  $\mathcal{NF}$-number of disjoint union of finite copies of complete graph, \textit{Ricerche math.} \textbf{74} (2025) 2759--2778.
		
		
		
		
		\bibitem{Faridi2002}
		S.~Faridi,
		\newblock The facet ideal of a simplicial complex,
		\newblock \emph{Manuscripta Math.} \textbf{109} (2002) 159--174.
		
		
		\bibitem{FranciscoHaVanTuyl2011}
		C.~A.~Francisco, H.~T.~H\`a, and A.~Van Tuyl,
		\newblock Coloring ideals, arrangements, and $\mathrm{Hilbert}$ functions of squarefree monomial ideals,
		\newblock \emph{J. Algebraic Combin.} \textbf{34} (2011) 309--322.
		
		\bibitem{GonzalezHoekstraMendoza}
		J.~Gonz\'alez and T.~I.~Hoekstra-Mendoza,
		\newblock On the homotopy type of complexes of graphs with bounded domination number,
		\newblock \emph{Contrib. Discrete Math.} \textbf{16}(3) (2021) 153--174.
		
		
		
		\bibitem{HerzogHibi2011}
		J.~Herzog and T.~Hibi,
		\newblock \emph{Monomial Ideals},
		\newblock Graduate Texts in Mathematics 260, Springer, 2011.
		
		\bibitem{HibiMahmood2020}
		T.~Hibi and H.~Mahmood,
		\newblock The NF-number of a simplicial complex,
		\newblock \textit{Algebra Colloquium} \textbf{29}(04) (2022) 643--650.
		
			\bibitem{Hochster1977}
		M.~Hochster,
		\newblock Cohen--Macaulay rings, combinatorics, and simplicial complexes,
		\newblock in \emph{Ring Theory II}, Lecture Notes in Pure and Applied Mathematics, Vol.~26, Marcel Dekker, 1977, pp.~171--223.
		
		\bibitem{Jonsson2008}
		J.~Jonsson,
		\newblock \emph{Simplicial Complexes of Graphs},
		\newblock Lecture Notes in Mathematics 1928, Springer, 2008.
		
		\bibitem{MillerSturmfels2005}
		E.~Miller and B.~Sturmfels,
		\newblock \emph{Combinatorial Commutative Algebra},
		\newblock Graduate Texts in Mathematics 227, Springer, 2005.
		
		\bibitem{bilal} B. A. Rather, Independent domination polynomial of comaximal graphs of commutative rings, \emph{Algebra colloquium} \textbf{33}(2) (2026) 243--258.
		
		\bibitem{Reisner1976}
		G.~A.~Reisner,
		\newblock Cohen--Macaulay quotients of polynomial rings,
		\newblock \emph{Advances in Math.} \textbf{21} (1976) 30--49.
		
		\bibitem{Stanley1975}
		R.~P.~Stanley,
		\newblock The upper bound conjecture and Cohen--Macaulay rings,
		\newblock \emph{Studies in Applied Math.} \textbf{54} (1975) 135--142.
		
		\bibitem{Stanley1996}
		R.~P.~Stanley,
		\newblock \emph{Combinatorics and Commutative Algebra},
		\newblock Second edition, Progress in Mathematics, Vol.~41, Birkh\"auser, 1996.
		
		\bibitem{Vizing1965}
		V.~G.~Vizing,
		\newblock An estimate of the external stability number of a graph,
		\newblock \emph{Dokl. Akad. Nauk SSSR} \textbf{164} (1965) 729--731.
		
		\bibitem{Wachs2003}
		M.~L.~Wachs,
		\newblock Topology of matching, chessboard, and general bounded degree graph complexes,
		\newblock \emph{Algebra Universalis} \textbf{49} (2003) 345--385.
		
		\bibitem{Villarreal2001}
		R.~H.~Villarreal,
		\newblock \emph{Monomial Algebras},
		\newblock Monographs and Textbooks in Pure and Applied Mathematics, Vol.~238, Marcel Dekker, 2001.
		
		\bibitem{West2001} D. B. West, \emph{Introduction to Graph Theory}, 2nd ed., Prentice Hall, Upper Saddle River, NJ, 2001.
		\bibitem{BrunsHerzog1998}
		W.~Bruns and J.~Herzog,
		\newblock \emph{Cohen--Macaulay Rings},
		\newblock Cambridge Studies in Advanced Mathematics, Vol.~39, Cambridge University Press, 1998.
		
		
		
		
		
		
		
		
	\end{thebibliography}
\end{document}